\documentclass[11pt,reqno]{amsart}
\usepackage{mathpazo}%provides access to standard palatino font for maths
\usepackage[utf8]{inputenc}
\usepackage{amsmath} 
\usepackage{amssymb, latexsym, stmaryrd, amsthm, dsfont, amsfonts, amsbsy,amsthm, amsmath, mathrsfs}%all the usual maths packages
\usepackage{tikz-cd}%for diagrams
\usepackage{mathtools}
\usepackage{bm}
\usepackage{enumitem}
\usepackage{tcolorbox}%package for creating coloured boxes of text
\allowdisplaybreaks

\usepackage[margin=1.2in]{geometry}
\usepackage[normal,sl,up,bf]{caption}%for captions outside of float enviroment

\usepackage{microtype}  
\usepackage[all, knot]{xy}
        \xyoption{arc} 
        \xyoption{web}                 %% for the best typographic result
\makeatletter                            %% A correction needed for microtype:
\def\MT@register@subst@font{\MT@exp@one@n\MT@in@clist\font@name\MT@font@list
 \ifMT@inlist@\else\xdef\MT@font@list{\MT@font@list\font@name,}\fi}
\makeatother
\usepackage[pdftex,bookmarks,bookmarksnumbered,linktocpage,   %%  customize to your liking
         colorlinks,linkcolor=blue,citecolor=blue]{hyperref}

\newtcolorbox{review}{colback=blue!5!white,colframe=blue!75!black}

\usepackage[backend=bibtex,style=numeric,natbib=true]{biblatex} % Use the bibtex backend with the authoryear citation style (which resembles APA)

\usepackage{etoolbox}% removal of lineskips for align and equation enviroments 
\newcommand{\zerodisplayskips}{%
  \setlength{\abovedisplayskip}{1.5pt}%
  \setlength{\belowdisplayskip}{1.5pt}%
  \setlength{\abovedisplayshortskip}{1.5pt}%
  \setlength{\belowdisplayshortskip}{1.5pt}}
\appto{\normalsize}{\zerodisplayskips}
\appto{\small}{\zerodisplayskips}
\appto{\footnotesize}{\zerodisplayskips}

\let\bigland=\bigwedge%alt big conjunction
\let\biglor=\bigvee%alt big disjunction
\let\isom=\cong%alt isomorphism
\newcommand\amp{\ensuremath{\ \&\ }}%better & symbol for math
\newcommand{\Luk}{\text{\textsl{\L}}}% slanted polish L (for text AND for math)

\let\phi=\varphi

\newcommand{\CM}{\mathcal{M}}

\newcommand{\CL}{\mathcal{L}}

\newcommand{\CP}{\mathcal{P}}

\theoremstyle{plain}
\newtheorem{theorem}{Theorem}[section]

\newtheorem{lemma}[theorem]{Lemma}
\newtheorem{corollary}[theorem]{Corollary}

\theoremstyle{definition}
\newtheorem{definition}[theorem]{Definition}

\newtheorem{example}[theorem]{Example}

\theoremstyle{remark}

\newtheorem{remarks}[theorem]{Remarks}

\begin{document}
\title[Many-valued Homomorphism Preservation]{Homomorphism Preservation Theorems for Many-valued Structures}

\author{James Carr}
\address{University of Queensland}
%\email{insert email}
%\date{\today}

\begin{abstract}
A canonical result in model theory is the homomorphism preservation theorem (h.p.t.)\ which states that a first-order formula is preserved under homomorphisms on all structures if and only if it is equivalent to an existential-positive formula, standardly proved via a compactness argument. Rossman (2008) established that the h.p.t.\ remains valid when restricted to finite structures. This is a significant result in the field of finite model theory. It stands in contrast to the other preservation theorems proved via compactness where the failure of the latter also results in the failure of the former~\cite{AjtaiGurevich87},~\cite{Tait59}. Moreover, almost all results from traditional model theory that do survive to the finite are those whose proofs work just as well when considering finite structures. Rossman's result is interesting as an example of a result which remains true in the finite but whose proof uses entirely different methods. It is also of importance to the field of constraint satisfaction due to the equivalence of existential-positive formulas and unions of conjunctive queries~\cite{ChandraMerlin77}. Adjacently, Dellunde and Vidal (2019) established a version of the h.p.t.\ holds for a collection of first-order many-valued logics, namely those whose structures (finite and infinite) are defined over a fixed finite MTL-chain.   

In this paper we unite these two strands. We show how one can extend Rossman's proof of a finite h.p.t.\ to a very wide collection of many-valued predicate logics. In doing so, we establish a finite variant to Dellunde and Vidal's result, one which not only applies to structures defined over algebras more general than MTL-chains but also where we allow for those algebra to vary between models. We identify the fairly minimal critical features of classical logic that enable Rossman's proof from a model-theoretic point of view, and demonstrate how any non-classical logic satisfying them will inherit an appropriate finite h.p.t. This investigation provides a starting point in a wider development of finite model theory for many-valued logics and, just as the classical finite h.p.t.\ has implications for constraint satisfaction, the many-valued finite h.p.t.\ has implications for valued constraint satisfaction problems. 
\end{abstract}

\maketitle

%\tableofcontents

%-------------------------------------------------------------------------

\section{Introduction}\label{Intro}
Preservation theorems are a group of results from model theory that describe relationships between syntactic and semantic properties of first-order formulas. Three fundamental preservation theorems dating from the 1950s are the \L os--Tarski theorem~\cite[Theorem 5.4.4]{Hodges97}, Lyndon's positivity theorem~\cite[Corollary 5.3]{Lyndon59} and the homomorphism preservation theorem (h.p.t.)~\cite[Theorem 1.3]{Rossman08}. Each states that a certain syntactic class of formulas contains (up to logical equivalence) all first-order formulas preserved under a particular kind of homomorphism, and each is standardly proved via compactness. Their place in traditional classical model theory naturally motivated an interest in their status in other model-theoretic contexts. Two in particular are finite model theory and many-valued logics. 

It was known for some time that many famous theorems about first-order logic fail when restricted to finite structures, including the bedrock compactness theorem~\cite[Chapter 0]{EbbinghausFlum95}. This failure suggested that theorems which utilised compactness in their proofs might also fail; indeed both the \L os--Tarski~\cite{Tait59} and Lyndon positivity theorem~\cite{AjtaiGurevich87, Gurevich84} fail in the finite. By contrast, Rossman (2008) established that the h.p.t.\ remains valid when restricted to finite structures. This result is an interesting example in the world of finite model theory, frequently the results and techniques from traditional model theory that do survive to the finite (such as Ehrenfeucht--Fra\"iss\'e games~\cite[Chapter 1]{EbbinghausFlum95}) are those whose proofs work just as well when considering only finite structures. With Rossman's result we have a theorem which remains true in the finite but whose proof uses entirely different methods.\footnote{Notably, Rossman's proof strategy can be used retroactively in the infinite setting. That is, one way to understand why the h.p.t.\ remains true in the finite is that whilst the h.p.t.\ is usually proved via compactness there are compactness free proof strategies which the finite setting brings to light. This is not simply a curiosity, it has substantive consequences. The compactness-free strategy provides a refinement of the standard h.p.t.\ to an 'equirank version' which Rossman also proves~\cite[Theorem 4.12]{Rossman08}. We will pick this up again in section ~\ref{ResiCon}.}\

A separate stand is the status of preservation theorems in non-classical logics - especially many-valued ones. Dellunde and Vidal (2019) established a version of the h.p.t.\ that holds for structures defined over a fixed finite MTL-chain. MTL-algebras provide the algebraic semantics for the monoidal t-norm logic MTL, a basic propositional fuzzy logic that encompasses the most well-studied fuzzy logics including H\'{a}jek's basic logic BL, G\"{o}del--Dummett logic G and \L ukasiewicz logic \L\ ~\cite[Chapter 1, Section 2]{Handbook}. MTL-chains are the finitely subdirectly irreducible MTL-algebras~\cite{Noguera:Thesis} and so provide a natural restricted class to work over that is better behaved~\cite{DellundeVidal19}. Accordingly, Dellunde and Vidal's result applies to a significant collection of first-order many-valued logics. The choice to work with a finite MTL-chain is far from arbitrary, rather it reflects a common pattern in research into non-classical logic:\begin{enumerate}
    \item Identify a potential theorem about classical logic to be generalised to non-classical settings, adjusting concepts involved if needed.
    \item Attempt a proof in a non-classical logic as weak as possible whilst being able to formulate the theorem in question.
    \item Strengthen the logic you work with, bringing it closer to classical logic, in accordance with problems arising in attempted proofs.
    \item Ideally, demonstrate that the aspects of classical logic adopted were necessary through counterexamples. 
\end{enumerate}
Rather than a strict ordering these steps feed into each other. For example, the setting of a proof will affect how classical concepts are appropriately generalised or exactly which generalised concepts one decides to work with affects what aspects of classical logic may be relevant. In a model-theoretic concept this process plays out especially with regard to the algebra that models are defined over. When it comes to Dellunde and Vidal's work, the essential components are a form of compactness theorem and a condition on models called witnessing; the importance of these properties in the essential proofs are what motivates the finite MTL-chain setting. 

Our investigation picks up at the meeting point of these two strands. We will demonstrate how one can extend Rossman's proof of a finite h.p.t.\ to a very wide collection of many-valued predicate logics. In doing so we establish a finite variant to Dellunde and Vidal's result, a variant which not only applies to structures defined over more general algebras than MTL-chains, but also where we allow for the algebra our structures are defined over to vary. The framework we choose to work in is deliberately chosen to be somewhat artificial in order to highlight the relevant behaviour that is necessary for the proof of the main theorem. It applies directly to a few different classes of algebras and corresponding models discussed in the literature, namely certain classes of residuated lattices of which MTL-chains are an example and the linearly-ordered integral abelian monoids which appear in discussions around valued constraint satisfaction problems ~\cite{HorcikMoraschiniVidal17}. The paper proceeds as follows. In Section~\ref{Prelim} we properly introduce our many-valued models including both the class of algebras we define our models over and the models themselves. We'll also introduce the appropriate generalisations of the essential model-theoretic concepts we are interested in. In Section~\ref{B&FSystem} we turn to an important piece of background theory - back-\&-forth systems. These have already been investigated in a many-valued context~\cite{DellundeGarciaNoguera18} and we recall this work whilst making small adaptations fit for our context. In Section~\ref{EPSent} we introduce existential-positive sentences in the many-valued context, noting that in their generalisation we see a splitting of one syntactic class into a collection of interrelated classes. Accordingly, we identify the appropriate generalisation of the h.p.t.\ in our many-valued context by proving the straightforward direction of the theorem, before turning to the proof of the main theorem in Section~\ref{FinHPT}. Finally, in Section~\ref{ResiCon} we consider variants to the main theorem relating to an alternative notion of conjunction and which ties our investigation directly to the work of Dellunde and Vidal. This also lets us say something about the potential status of homomorphism preservation theorems in the related topic of semiring semantics. We conclude with some remarks about further study with a particular eye towards other plausible preservation theorems.

%-----------------------------------------------------------------------

\section{Related Work}
Our work sits firmly within the wider project of developing finite model theory for many-valued logics begun in recent years. In the same fixed and finite MTL-chain setting as Dellunde and Vidal's work there have been investigations into the status of other preservation theorems. These have been quite successful, including a proof of a generalised \L os--Tarski theorem and a Chang--\L os--Suszko theorem which links $\forall_2$ formulas to preservation under unions of chains~\cite{BadiaCostaDellundeNoguera19}. Much of this work is built on the generalisation of fundamental concepts of classical model theory to the non-classic setting. This includes understanding the different types of maps one can define between many-valued models~\cite{DellundeGarciaNoguera16}, the method of diagrams~\cite{Dellunde12}, back-and-forth systems~\cite{DellundeGarciaNoguera18} and L\"{o}wenheim--Skolem theorems~\cite{DellundeGarciaNoguera16}. Other themes explored in the literature include developing appropriate 0-1 law equivalents in the many-valued setting~\cite{BadiaNoguera22}.

Directly adjacent to our main interest in homomorphism preservation theorems in finite models is the field of constraint satisfaction problems. The classical constraint satisfaction problem (CSP) is closely connected with classical finite-model theory as finite databases are naturally represented by classical finite models in relational signatures. Moreover, there is an equivalence between existential-positive formulas and unions of conjunctive queries~\cite{ChandraMerlin77}, the syntactic class and query class that the h.p.t.\ and CSP are respectively concerned with. This means Rossman's result is of direct consequence for the CSP field~\cite{Rossman08}. Similarly, Dellunde and Vidal's investigation is further motivated by the application of models defined over MTL-algebras to \textit{valued} constraint satisfaction problems (VCSP). A generalisation of classical CSP, in a valued constraint satisfaction problem the constraints are assigned some form of weighting which is optimised in the solution. This has been effectively modelled by taking the weights as elements of an algebra and utilising the algebraic operations to interpret their combination in a potential solution~\cite{BistarelliMontanariRossi97}. 

Another literature adjacent to (finite) model theory for many-valued logics is the literature on semiring semantics and its applications in database theory ~\cite{GreenTannen17}. Within this literature the basic objects of study are defined in the same way as in the predicate many-valued logic literature, however there the relationship between semantics and syntax in focus is different. In particular, in the setting of semirings the quantifiers $\forall$ and $\exists$ are treated differently providing an alternative generalisation of the classical notions. The treatments do overlap in the case of lattice semirings and our consequently our investigations is relevant for models defined over these semirings. Much of the core theory of classical model theory has been explored in the context of semiring semantics including Ehrenfeucht-Fra\"{i}ss\`{e} games ~\cite{BrinkeGradelMrkonjic23}, 0-1 laws ~\cite{GradelHelalNaafWilke22} and locality ~\cite{BiziereGradelNaaf23}. The most pertinent to our investigation is the topic of conjunctive queries over semiring annotated databases ~\cite{Green11}.

%-----------------------------------------------------------------------

\section{Preliminaries}\label{Prelim}
We start by introducing the structures we are concerned with. These will consist of three components, a non-empty set equipped with interpretations of non-logical symbols akin to the familiar classical model; an algebra that defines the truth-values of formulas and a subset of the algebra that determines which truth-values are considered `true' or `designated'. As we will see, the choice of algebras we define our structures over is intimately tied with the resulting properties of the structures. The class of algebras we work with have a fairly artificial definition; the intention is they provide a level of generality that can be applied cleanly to a multitude of cases. In practice different motivations lead one to be concerned with more specific algebras than what we introduce here.

\begin{definition}\label{def:interplat}
    An \emph{interpreting lattice} is a pair $(A,F)$ where $A$ is an algebra in signature $\CL$ with $\{\lor,\land\}\subseteq \CL$ and $F\subseteq A$ such that:\begin{itemize}
        \item $\langle A,\land,\lor\rangle$ is a lattice;
        \item $\forall a,b\in A$ $a\land b\in F$ iff $a\in F$ and $b\in F$;
        \item $\forall a,b\in A$ $a\lor b\in F$ iff $a\in F$ or $b\in F$. 
    \end{itemize}
\end{definition}
\begin{remarks}
    When we introduce our first-order semantics the role of the algebra $A$ will be to provide the meaning of truth-values and the logical connectives $\circ\in\CL$, whilst the role of the subset $F$ is to give meaning to the $\models$ relation. The two conditions on the subset $F$ are equivalent to requiring that $F$ is a prime filter of $A$. Later we will see that this is an essential restriction (example ~\ref{primerequirement}).

    Further, observe that whilst we require the presence of the connectives $\lor$ and $\land$ these need not exhaust the algebraic signature.
\end{remarks}

\begin{example}\label{3maincases}
    This setting encompasses a number of the algebras that many-valued models have been investigated over in the literature. Naturally, the two element Boolean algebra $\mathbf{2}$ which is in signature $(\lor,\land,\neg,\bot,\top)$ is an interpreting lattice taking $F=\{\top\}$. Otherwise we can place our examples into three broad groups.
    
    The starting motivation of this investigation comes from the literature on predicate fuzzy logics and within that literature a natural example of interpreting lattices emerges from bounded, prelinear, residuated lattices, also known as UL-algebreas. UL-algebras, which have signature $(\lor,\land,\&,\rightarrow,0,1)$ are the algebraic semantics of the (propositional) uninorm logic and are discussed in the context of many-valued models in~\cite{DellundeGarciaNoguera18}. The usual understanding of truth is to take $F$ as the upset (in the lattice order) of the unit $1$ of the residuated conjunction $\&$, i.e. $F=\{a\in A:a\geq 1\}$. Subdirectly irreducible UL-algebras are linearly ordered, (and referred to as UL-chains), and attention is often restricted to models defined over UL-chains rather than all UL-algebras~\cite{DellundeGarciaNoguera18}. Our setting follows this restriction, paired with the aforementioned filter $F$, every UL-chain is an interpreting lattice. The previously mentioned MTL-algebras are the algebraic semantics for the monoidal t-norm logic and a homomorphism preservation theorem for (possibly infinite) models defined over a \textit{fixed} and \textit{finite} MTL-\textit{chain} was established in~\cite{DellundeVidal19}. They are the integral restrictions of UL-algebras meaning that the unit of residuated conjunction is also the top element of the lattice. MTL-chains themselves encompass the standard algebras of the three core fuzzy logics, \L ukasiewicz logic, G\"{o}del-Dummet logic and product logic. These all take the real unit interval $[0,1]$ as their domain, interpret the lattice connecives $\lor,\land$ as expected and the residuated conjunction $\&$ as $a\odot b=max\{0,a+b-1\}$, $\land$ and real multiplication $\cdot$ respectively. The implication connective $\rightarrow$ is the residuum of the corresponding residuated conjunction $a\rightarrow b=max\{c\in[0,1]:a\cdot c\leq b\}$, and each is an interpreting lattices under the prime filter $F=\{1\}$:\[\Luk=\langle [0,1],\lor,\land,\odot,\rightarrow,0,1\rangle\;\; G=\langle [0,1],\lor,\land,\land,\rightarrow,0,1\rangle\;\; P=\langle [0,1],\lor,\land,\cdot,\rightarrow,0,1\rangle.\]
    
    MTL-algebras are one of the potential candidates for interpreting valued constraint satisfaction problems. Another class of algebras discussed in the VCSP context are linearly-ordered integral abelian monoids ($l$-monoids)~\cite{HorcikMoraschiniVidal17}. These are in signature $(\lor,\land,\cdot,0,1)$ are also covered by interpreting lattices, again taking $F$ as the upset of the unit $1$ of the monoid operator $\cdot$. Note that the three standard algebras on the unit interval defined above are also linearly-ordered integral abelian monoids when we forget about the implication connective $\rightarrow$. The most general approach towards the study of VCSP complexity takes c-semirings as their evaluating algebras~\cite{BistarelliMontanariRossi97}. These are in the signature $(+,\times,0,1)$ and are also integral in the above sense, but are not in general interpreting lattices. They do not have an operator in the signature corresponding to the lattice infimum, instead their notion of conjunction is just the monoid operate $\times$. Due to their integrality however there is something we can say, and we return to consider the integral setting in Section~\ref{ResiCon}.

    The application of semirings as a generalisation to classical model theory is a significant branch of literature in its own right. Within that literature there is less focus on understanding a form of the $\models$ relation and designating certain values of the algebra as true as the subset $F$ is intended to do ~\cite{GreenTannen17}. Three important semirings, the \L ukasiewicz semiring, min-max semiring and tropical semiring, are the $\lor,\&$ reducts of the aforementioned MTL-chains. As with c-semirings they need not be interpreting lattices even when the underlying natural order is a lattice order as the operator corresponding to conjunction may not reflect the lattice conjunction. Nevertheless, any lattice semiring paired with a prime filter is an interpreting lattice. In particular, the min-max semiring $(S,min,max,s,t)$ where $s$ and $t$ are the least and greatest element respectively of a totally ordered set is an interpreting lattice when paired with any upset.
\end{example}

We define our many-valued structures over interpreting lattices. We start with introducing the syntax which is defined essentially as classically except we allow for any algebraic connectives present in the signature $\CL$. We will only introduce \emph{relational languages} as our main investigation is restricted to these. All the basic definitions in this section extend naturally to a language complete with function symbols in the expected manner.

\begin{definition}
    A (relational) \textit{predicate language} $\CP$ is a is a non-empty set $P$ of \textit{predicate symbols} all of which comes with an assigned natural number that is its \textit{arity}. The predicate symbols of arity zero are called \textit{truth constants}.

    Given an algebraic signature $\CL$ such that $\{\lor,\land\}\subseteq \CL$ and a countable set of variables $Var$, We define the set of $\CL,\CP$-\textit{formulas}, denoted $Fm(\CL,\CP)$ or simply $Fm$, inductively as in classical logic. That is the set of $\CL,\CP$-formulas is the minimum set $X$ such that:\begin{itemize}
        \item $X$ contains all expressions of the form $P(x_1,...,x_n)$ where $P$ is an $n$-ary predicate symbol and $x_1,...,x_n$ are variables. These are the \textit{atomic} $\CL,\CP$-formulas
        \item $X$ is closed under the propositional connectives of $\CL$ (in particular it contains all the constant symbols of $\CL$).
        \item If $\phi\in X$ and $x\in Var$ then $\forall x\phi$ and $\exists x\phi\in X$. 
    \end{itemize}

    We define the notions of free occurrence of a variable, open formula, substitutability and sentence as in classical model theory.
\end{definition}

\begin{definition}
    Let $\CP$ be a predicate language and $\CL$ an algebraic signature. We define a $\CL,\CP$-structure $\CM$ as a triple $(A,F,M)$ where $(A,F)$ is an interpreting lattice in signature $\CL$ and $M$ is a pair $(M,\langle P^M\rangle_{P\in\CP})$ where:\begin{itemize}
        \item $M$ is a non-empty set,
        \item $P^M:M^n\rightarrow A$ for each $P\in\CP$ ($P^M\in A$ when $n=0$). 
    \end{itemize}
    The set $M$ is called the \textit{domain} of $\CM$ and the mappings $P^M$ are called the interpretations of predicate symbols in $\CM$.
\end{definition}
\begin{remarks}
    In the definition of structure we see an important difference between a algebraic constant $d\in\CL$ and an interpreted nullary predicate symbol for $P\in\CP$. The interpretation of the former is determined by the interpreting lattice $(A,F)$ and takes the same truth value $d\in A$ for each structure $(A,F,M)$. The latter does not come with a fixed interpretation on the algebra, each structure assigns it some truth value $P^M\in A$.
\end{remarks}

\begin{definition}
    As in classical logic, we define an $\CM$-valuation of the object variables is a mapping $v$ from the set of object variables $Var$ into $M$. Given an $M$-valuation $v$, an object variable $x$ and an element $m\in M$ we denote by $v_{x=m}$ the $M$-valuation defined as:\[v_{x=m}:=\begin{cases}
        m & \text{if }y=x\\
        v(y) & \text{otherwise}
    \end{cases}\]
    Let $\CM$ be a $\CL,\CP$-structure. We define the \textit{truth values} $||-||^M_v$ of the $\CL,\CP$-formulas for a given $M$-valuation $v$ recursively as follows:\begin{align*}
        &||P(x_1,..,x_n)||^M_v =P^M(v(x_1),...,||v(x_n)),\text{ for each }n\text{-ary }P\in \CP,\\
        &||\circ(\phi_1,...,\phi_n)||^M_v =\circ^A(||\phi_1||^M_v,...,||\phi_n||^M_v),\text{ for each }n\text{-ary }\circ\in\CL,\\
        &||\forall x\phi||^M_v =inf_{\leq}\{||\phi||^M_{v_{x=a}}:a\in M\},\\
        &||\exists x\phi||^M_v =sup_{\leq}\{||\phi||^M_{v_{x=a}}:a\in M\}.
    \end{align*}
    If the infimum or supremum does not exist we take the truth-value of such a formula to be undefined. We say that a $\CP$-structure is \textit{safe} iff $||\phi||^M_v$ is defined for every formula and every valuation, and reserve the term $\CP$-model for safe $\CP$-structures. We will restrict attention to $\CP$-models.

    We denote by $||\phi||^M=a$ for $a\in A$ that $||\phi||^M_v=a$ for all $M$-valuations $v$.
\end{definition}
\begin{remarks}\label{quantifierremark}
    The interpretation of the quantifiers $\exists$ and $\forall$ deserves some close attention. Defined relative to interpreting lattices, the quantifiers capture both the idea of a supremum/infinmum relative to the natural order on the algebra and act as generalised versions of the algebraic connectives $\lor$ and $\land$. These are really two distinct ways of approaching quantifiers for models defined over algebras. The former is the route taken in the literature on predicate many-valued logics and remains meaningful even for algebras in signatures without a form of disjunction/conjunction but a sensible notion of order, e.g. the purely implication signatures of the weakly implicative logics ~\cite[Chapter 7]{CintulaNoguera21}. The latter is the route taken in the literature on semiring semantics and in non-lattice semirings the quantifiers have a vastly different behaviour with respect to the underlying order on the semiring ~\cite{GreenTannen17}. 
\end{remarks}

\begin{example}\label{exa:struc}
    We borrow an example from ~\cite[Example 20]{HorcikMoraschiniVidal17} to illustrate the form of these structures and which provides a VCSP formulation of the $\langle s,t\rangle$-MIN-CUT problem. We take $\CL=\{\land,\lor,0,1\}$ and $(A,F)$ the product interpreting lattice $P$ over the real unit interval (without implication) discussed in examples ~\ref{3maincases}. We take $\CP=\{R,P_s,P_t\}$ as the language of weighted (s,t)-cut graphs, $R$ is binary relation symbol and $P_s,P_t$ are unary relation symbols. We define the $\CL,\CP$-model $B_\alpha=\langle\{a,b\},R,P_s,P_t\rangle$ where $P^B_s(a)=1,P^B_s(b)=0,P^B_t(a)=0, P^B_t(b)=1$ and $R^B$ is defined for $0<\alpha<1$ by:\[\begin{tikzcd}
a \arrow["1"', loop, distance=2em, in=215, out=145] \arrow[r, "\alpha", bend left] & b \arrow[l, "1", bend left] \arrow["1"', loop, distance=2em, in=35, out=325]
\end{tikzcd}\]
\end{example}

The valuation function $||-||$ is the fundamental relationship between semantics and syntax enabled by these objects. One can certainly study many-valued models entirely from this point of view, indeed this is the approach taken in the semiring semantic literature (e.g. ~\cite{BrinkeGradelMrkonjic23,BiziereGradelNaaf23}. However, there is also motivation to introduce a notion of validity and logical consequence based on designated truth-values, this is the usual approach in the literature on predicate fuzzy logics (e.g. ~\cite{BadiaCostaDellundeNoguera19,DellundeVidal19}. With a definition of validity in place many of the standard concepts from model theory can be defined just as they are in classically.  e.g.\ the theory of a model is still just all the sentences that it models $Th(A,F,M):=\{\phi\in Fm(\CL,\CP):(A,F,M)\models\phi\}$. Others require adjustments and generalisations.
\begin{definition}
    Given a set of sentences $\Phi$ we say that $\CM$ is a \textit{model of} $\Phi$, denoted $(A,F,M)\models\Phi$, iff for every $\phi\in\Phi$ $||\phi||^M\in F$.
    
    We say that two $\CL,\CP$-models $(A,F,M),(B,G,N)$ are \textit{elementarily equivalent}, denoted $(A,F,M)\equiv (B,G,N)$ iff $Th((A,F,M))=Th(B,G,N)$. 

    We say that two $\CP$-models $M,N$ defined over the same interpreting lattice $A$ are \textit{strongly elementarily equivalent}, denoted $M\equiv^s N$ iff for all $\phi\in Fm$ $||\phi||^M=||\phi||^N$.
\end{definition}

We can only sensibly compare models defined over interpreting lattices in the same signature $\CL$. Additionally, later we will want to make further restrictions on the algebras we define models over to compare and contrast the resulting model theoretic behaviour. It thus proves convenient to consider models defined over specified classes of algebras. 

\begin{definition}
    Let $\Phi$ be a $\CP$ theory and $K$ a class of interpreting lattices in the same signature $\CL$. We define the following sets of $\CP$-models.\begin{align*}
    Mod^K(\Phi)&:=\{(A,F,M):(A,F)\in K,M\text{ a safe structure over }A\text{ and }(A,F,M)\models\Phi\}.\\
    Mod^K_{fin(\Phi)}&:=\{(A,F,M):(A,F)\in K,M\text{ a safe finite structure over }A\text{ and }(A,F,M)\models\Phi\}.
\end{align*}
When $K=\{(A,F)\}$ we denote this $Mod^{A,F}(\Phi)$ and $Mod^{A,F}_{fin}(\Phi)$. When $K$ is the entire class of interpreting lattices (in a given signature $\CL$) we omit reference to it. 
\end{definition}

When we consider maps between structures we see a more radical departure from the classical case with our previous single definition of morphism shattering into a number of related notions. The landscape of maps between many-valued models is discussed in both~\cite{DellundeGarciaNoguera16} and~\cite{DellundeGarciaNoguera18} along with some initial outlining of the relationships between the different maps and what motivates them. We highlight four types of map (along with some additional properties they may have). First and foremost are protomorphisms which we introduce here and are motivated by our current investigations. 
\begin{definition}\label{protomorph}
    Let $(A,F,M),(B,G,N)$ be $\CP$-models. A map $g\colon M\rightarrow N$ is a \textit{protomorphism} from $(A,F,M)$ to $(B,G,N)$ iff:\begin{align*}
        &\text{for every }R\in \CP\text{ and }\bar{m}\in M\ R^M(\bar{m})\in F\text{ implies }R^N(g(\bar{m}))\in F.
    \end{align*}
\end{definition}

We briefly mention three other kinds of maps that will be relevant to our investigation. Two of them, homomorphisms and strong homomorphisms are pre-existing in the literature~\cite{DellundeGarciaNoguera16}. Monomorphisms are the natural one-directional weakening of strong homomorphisms.
\begin{definition}\label{othermorph}
    Let $f\colon A\rightarrow B$ and $g\colon M\rightarrow N$ be maps. We call the pair $(f,g)\colon (A,F,M)\rightarrow (B,G,M)$ a \textit{homomorphism} from $(A,F,M)$ to $(B,G,N)$ iff:\begin{itemize}
        \item $f$ is a $\CL$-homomorphism, i.e.\ it preserves the algebraic connectives of $A$. 
        \item $g$ is a protomorphism from $(A,F,M)$ to $(B,G,N)$. 
    \end{itemize}

    Let $(f,g)\colon (A,F,M)\rightarrow (B,G,N)$ be a homomorphism. We say that $(f,g)$ is a:\begin{itemize}
        \item \emph{monomorphism} iff for every  $R\in\CP$ and $\bar{m}\in M\;f(R^M(\bar{m}))\leq R^N(g(\bar{m}))$.
        \item \textit{strong homomorphism} iff for every  $R\in \CP$ and $\bar{m}\in M\;f(R^M(\bar{m}))=R^N(g(\bar{m}))$.
    \end{itemize}
\end{definition}
\begin{remarks}
    In the classical case a number of these concepts collapse together. Working over a fixed algebra means that any protomorphism is a homomorphsim when paired with the identity map. In fact, because we only have two elements in the algebra the identity the only available homomorphism of algebras and homomorphisms and monomorphisms coincide.  
    
    Homomorphisms and monomorphisms are the most direct attempt to generalise the classical notion of homomorphism, the former caring about the preservation of true relations and the latter caring about the the valuation of relations directly. Strong homomorphisms, first introduced in~\cite{NolaGerla86}, come from a categorical lens. They preserve the structure more rigidly demanding equality of all relations rather than simply preserving truth. Protomorphisms are a weakening of homomorphisms. They emerge from the observation that there really is no interaction between the algebra and structure map in the definition of homomorphism - we simply require the algebra map to be present. This naturally suggests considering the weaker notion where we do not require its presence, especially important when considering models defined over algebras $A$ and $B$ when $Hom(A,B)=\varnothing$. We will comment on the nature of preservation theorems for both homomorphisms and monomorphisms in our concluding remarks ~\ref{Conclusion}.
\end{remarks}

\begin{definition}
    The existence of a given map between $\CP$-models is a relation on the class of $\CP$-models and denoted as follows:\begin{align*}
        (A,F,M)\rightarrow_p(B,G,N)&: \text{protomorphism}\\
        (A,F,M)\rightarrow (B,G,N)&: \text{homomorphism}\\
        (A,F,M)\rightarrow_m(B,G,N)&: \text{monomorphism}\\
        (A,F,M)\rightarrow_s(B,G,N)&:\text{strong homomorphism}
    \end{align*}
    We use $\leftrightarrows_{x}$ to denote the existence of a map of the corresponding type between $\CP$-models in both directions.
\end{definition}

%------------------------------------------------------------
\section{Back-and-forth systems}\label{B&FSystem}
Before we begin our investigation into the finite homomorphism preservation theorem for our many-valued models we first need to recover an important model-theoretic tool - back-and-forth systems. Back-and-forth systems have already been covered in a many-valued setting in~\cite{DellundeGarciaNoguera18} specifically for models defined over UL-algebras. An effectively immediate generalisation of these proofs apply to models defined over interpreting lattices. We include the details here for direct reference.
\begin{definition}
    Let $(A,F,M),(B,G,N)$ be $\CP$-models. We say a partial mapping $(p,r)\colon(A,M)\rightarrow (B,G,N)$ is a \textit{partial isomorphism} from $(A,F,M)$ to $(B,G,N)$ iff:\begin{enumerate}
        \item $p$ and $r$ are injective.
        \item for every $\circ\in\CL$ and $\bar{a}\in A$ such that $\bar{a}\in dom(p)$ and $\circ(\bar{a})\in dom(p)$ \[p(\circ^A(\bar{a}))=\circ^B(p(\bar{a})).\]
        \item for every $R\in\CP$ and $\bar{m}\in M:\bar{m}\in dom(r)$ \[p(R^M(\bar{m}))=R^N(r(\bar{m})).\]
    \end{enumerate}
    $(A,F,M)$ and $(B,G,N)$ are said to be \textit{finitely isomorphic}, denoted $(A,F,M)\isom_f(B,G,N)$ iff there is a sequence $\langle I_n:n\in \omega\rangle$ with the following properties:\begin{enumerate}
        \item Every $I_n$ is a non-empty set of partial isomorphisms from $(A,F,M)$ to $(B,G,N)$ and $I_{n+1}\subseteq I_n$.
        \item Forth R: for every $(p,r)\in I_{n+1}$ and $m\in M$ $\exists (p,r')\in I_n:r\subseteq r'$ and $m\in dom(r')$.
        \item Back R: for every $(p,r)\in I_{n+1}$ and $n\in N$ $\exists (p,r')\in I_n:r\subseteq r'$ and $n\in im(r')$.
        \item Forth L: for every $(p,r)\in I_{n+1}$ and $a\in A$ $\exists (p',r)\in I_n:p\subseteq p'$ and $a\in dom(p')$.
        \item Back L: for every $(p,r)\in I_{n+1}$ and $b\in B$ $\exists (p',r)\in I_n:p\subseteq p'$ and $b\in im(p')$.
    \end{enumerate}
    We say $(A,F,M)$ and $(B,G,N)$ are $n$-finitely isomorphic, denoted $(A,F,M)\isom_n(B,G,N)$ iff there is a sequence $\langle I_m:m\leq n\rangle$ with the preceding properties. 
\end{definition}

We give the main definition of partial isomorphism for an arbitrary predicate language but the results we recover only apply to \emph{relational languages}, that is languages which only contain relation symbols (including those of 0-arity). Accordingly, from here on we restrict our attention to relational languages. 

Alongside this generalised notion of partial isomorphism and back-and-forth system we need a measure to compare with elementary equivalence. In classical logic this is quantifier depth but here we need something finer grained. 
\begin{definition}
    Let $\CP$ be a predicate language. Given a $\CP$-formula $\phi$ we define by induction the \textit{nested rank} of $\phi$, denoted $nr(\phi)$ as follows:\begin{itemize}
        \item If $\phi$ is atomic $nr(\phi)=0$.
        \item For every $\circ\in\CL$ and $\phi_1,...,\phi_k$ $nr(\circ(\phi_1,...,\phi_k))=1+\sum\limits_{i=1}^k nr(\phi_i)$
        \item For any $\CP$-formula $\phi(x)$ $nr(\forall x\phi)=nr(\exists x\phi)=nr(\phi)+3$.
    \end{itemize}
    We say two $\CP$-models $(A,F,M)$ and $(B,G,N)$ are $n$-elementarily equivalent, denoted $(A,F,M)\equiv_n(B,G,N)$ iff they are models of the same $\CP$-sentences up to nested rank $n$.
\end{definition}

Now the proof of~\cite[Theorem 24]{DellundeGarciaNoguera18} yields the corresponding theorem.
\begin{theorem}\label{mixedfintieb&f}
    Let $\CP$ be a finite relational language and $(A,F,M)$, $(B,G,N)$ be $\CP$-models and $n\in\omega$. Suppose $(A,F,M)\isom_n(B,G,N)$. Then $(A,F,M)\equiv_n(B,G,N)$. 
\end{theorem}

In the fixed algebra case we may restrict out attention to partial isomorphism where the left component is the identity. This can be convenient as it brings us closer to the classical setting, allowing us to drop back from nested rank to quantifier depth. This motivates the following definitions.\footnote{May is the operative word here. One may still apply the general case and construct partial isomorphisms where the algebra component is not the identity. This requires working with nested rank again.} We fix an interpreting lattice $(A,F)$ and we use $M$ to denote the $\CP$-model $(A,F,M)$.

\begin{definition}\label{fixedb&fsys}
    Let $M,N$ be $\CP$-models. We say a partial mapping $r:M\rightarrow N$ is a \textit{simple partial isomorphism} from $M$ to $N$ iff:\begin{enumerate}
        \item $r$ is injective
        \item for every function symbol $F\in\CP$ and $\bar{m}\in M:\bar{m},F^M(\bar{m})\in dom(r)$\begin{align*}
            r(F^M(\bar{m}))=F^N(r(\bar{m})).
        \end{align*}
        \item for every predicate symbol $R\in \CP$ and $\bar{m}\in M:\bar{m}\in dom(r)$\begin{align*}
            R^M(\bar{m})=R^N(r(\bar{m})).
        \end{align*}
    \end{enumerate}
    $M$ and $N$ are said to be \textit{simply finitely isomorphic}, denoted $M\isom_f N$ iff there is a sequence $\langle I_n:n\in\omega\rangle$ with the following properties:\begin{enumerate}
        \item Every $I_n$ is a non-empty set of partial isomorphisms from $M$ to $N$ and $I_{n+1}\subseteq I_n$
        \item Forth: For every $r\in I_{n+1}$ and $m\in M$ $\exists r'\in I_n:r\subseteq r'$ and $m\in dom(r')$.
        \item Back: For every $r\in I_{n+1}$ and $n\in N$ $\exists r'\in I_n:r\subseteq r'$ and $n\in im(r')$.
    \end{enumerate}
    We say $M$ and $N$ are $n$-finitely isomorphic, denoted $M\isom_nN$ iff there is a sequence $\langle I_m:m\leq n\rangle$ with the preceding properties.
\end{definition}

We then define $n$-elementary equivalence between for models $M$ and $N$ relative to $\CP$-formulas with \textit{quantifier-depth} at most $n$ rather than nested rank at most $n$ (quantifier depth defined in the usual way for classical logic). Following exactly the same proof as~\cite[Theorem 22]{DellundeGarciaNoguera18} yields the corresponding theorem:
\begin{theorem}\label{fixedfinteb&f}
    Let $\CP$ be a finite relational language, $M$, $N$ be $\CP$-models and $n\in\omega$. Suppose $M\isom_n N$. Then $M\equiv^s_n N$. 
\end{theorem}

%-----------------------------------------------------------

\section{Existential-Positive Sentences}\label{EPSent}
Now we turn to the actual work of proving finite homomorphism preservation. First we must make more precise the question we are asking. In the classical setting, the syntactic class of formulas we are interested in are the existential-positive formulas defined as formulas built from only the existential quantifier, disjunction and conjunction ($\exists$, $\lor$, $\land$). They are characterised as disjunctions of so called primitive positive sentences ~\cite{Rossman08}, formulas built from only $\exists$ and $\land$ with a normal form:\[\exists\bar{x}\bigland\limits_{i=1}^nR_i(\bar{x_i}),\]
where $\bar{x_i}$ denotes the subset of the free variables $\bar{x}$ that $R$ takes as argument. 

The characterisation relies on the distributivity of $\lor$ and $\exists$ over $\land$ in classical logic. When introducing interpreting lattices we allowed for potentially non-distributive lattices and thus lose the characterisation in general. Accordingly, we define existential-positive sentences directly as those sentences in the canonical form. We also highlight that the conjunction used is the $\land$ connective. 

\begin{definition}\label{elandpsentence}
    Let $\CP$ be a predicate language. We say a $\CP$-sentence $\phi$ is:\begin{enumerate}
        \item $\land$-primitive iff $\phi=\exists\bar{x}\bigland\limits_{i=1}^nR_i(\bar{x_i})$.
        \item existential-$\land$-positive iff $\phi=\biglor\limits_{i=1}^m\psi_i$ where each $\psi_i$ is a $\land$.p-sentence.
    \end{enumerate}
\end{definition}
\begin{remarks}
    In the context of \emph{finite} models over \emph{distributive} interpreting lattices we recover the  various distribution laws regarding $\land$, $\lor$ and $\exists$, when it comes to the quantifier it follows because we may essentially view $\forall$ and $\exists$ as conjunction and disjunction respectively.
    \[||\forall x\phi||^M_v=\bigland\{||\phi||^M_{v[x\mapsto m]}:m\in M\}\;\; ||\exists x\phi||^M_v=\biglor\{||\phi||^M_{v[x\mapsto m]}:m\in M\}.\]
    Using this and distributivity between $\land$ and $\lor$ it is straightforward (if tedious) to check that we have the desired distribution laws for $\exists$ with $\land$ and $\lor$.
    \begin{align*}
        ||\exists x\exists y(\phi(x)\land\psi(y))||&=||\exists x \phi(x)\land\exists y\psi(y)||\\
        ||\exists x\exists y(\phi(x)\lor\psi(y))||&=||\exists x\phi(x)\lor\exists y\psi(y)||
    \end{align*}
\end{remarks}

This outlines the syntax we are interested in. Given that we've introduced a number of different maps between models, the question becomes which pairings are viable candidates. We introduce two main homomorphism preservation theorems, one for the mixed algebra case and another for the fixed algebra case.
\begin{theorem}[Finite Protomorphism Preservation Theorem]\leavevmode\\
    Let $\CP$ be a predicate language and $\phi$ a consistent $\CP$-sentence in the finite (i.e.\ $Mod_{fin}(\phi)\not=\varnothing)$.
    Then $\phi$ is equivalent in the finite to an e.$\land$.p-sentence $\psi$ iff $\phi$ is preserved under protomorphisms. That is there exists an e.$\land$.p-sentence $\psi:Mod_{fin}(\phi)=Mod_{fin}(\psi)$ iff $Mod_{fin}(\phi)$ is closed under $\rightarrow_p$. 
\end{theorem}
\begin{theorem}[Fixed Finite Homomorphism Preservation Theorem]\leavevmode\\
    Let $\CP$ be a predicate language, $(A,F)$ an interpreting lattice and $\phi$ a consistent $\CP$ sentence over $(A,F)$ in the finite (i.e.\ $Mod^{A,F}_{fin}(\phi)\not=\varnothing$). The following are equivalent:\begin{enumerate}[label=\roman*.]
        \item $\phi$ is equivalent over $(A,F)$ in the finite to an e.$\land$.p sentence $\psi$, i.e.\ there is an e.$\land$.p-sentence $\psi:Mod^{A,F}_{fin}(\phi)=Mod^{A,F}_{fin}(\psi)$.
        \item $\phi$ is preserved under protomorphisms on $(A,F)$, i.e.\ $Mod^{A,F}_{fin}(\phi)$ is closed under $\rightarrow_p$.
        \item $\phi$ is preserved under homomorphisms on $A$, i.e.\ $Mod^{A,F}_{fin}(\phi)$ is closed under $\rightarrow$.
    \end{enumerate}
\end{theorem}
\begin{remarks}
    Observe that taking $(A,F)=(\mathbf{2},\{\top\})$ this statement is exactly Rossman's finite homomorphism preservation theorem. 
\end{remarks}

Viable here means that the formulas of the type referred to are actually preserved by the maps in question. Proving that the e.$\land$.p and protomorphism pairing is viable is a straightforward induction driven by the semantic properties of $\exists,\land$ and $\lor$ we have already encountered. Whilst straightforward it is worth belabouring the point a little to make clear exactly what behaviour of interpreting lattices is driving this. Namely, by considering finite models the behaviour of $\exists$ reduces to a large disjunction across all elements of the domain. Combined with the nice behaviour of $F$ with respect to $\land$ and $\lor$ from definition ~\ref{def:interplat}, the truth of formulas leading with an existential quantifier is \emph{witnessed} by some element of the domain, that is for any finite $\CP$-model $(A,F,M)$ and $\CP$-formula $\exists x\phi(x)$:\[||\exists x\phi(x)||^M\in F\text{ iff }\exists m\in M:||\phi(m)||^M.\] Putting this all together, the validity of an e.$\land$.p-sentence in a model is entirely determined at the atomic level and independent of the given interpreting lattice. 
\begin{lemma}\label{easylemma}
    Let $(A,F,M)$ be a finite $\CP$-model and $\psi=\biglor\limits_{i=1}^n\theta_i$ an e.$\land$.p-sentence. Then $(A,F,M)\models\psi$ iff $\exists 1\leq i\leq n:\exists\bar{m}\in M:$ for all atomic subformula $R_j(\bar{x_j})\subseteq\theta_i$ $R^M_j(\bar{m_j})\geq 1_A$, where $\bar{m_j}$ denotes the subsequence of $\bar{m}$ corresponding to the subsequence $\bar{x_j}$ of $\bar{x}$. 
\end{lemma}
\begin{proof}
    \begin{align*}
        &(A,M)\models\psi\text{ iff }||\psi||^M\in F\text{ iff }||\biglor\limits_{i=1}^n\theta_i||^M\in F &\\
        &\text{ iff } \exists 1\leq i\leq n:||\theta_i||^M\in F &\text{ (def ~\ref{def:interplat})}\\
        &\text{ iff }\exists 1\leq i\leq n:||\biglor\limits_{\bar{m}\in M}\bigland\limits_{j=1}^M R_j(\bar{x_j})||^M_{v[\bar{x}\mapsto\bar{m}]}\in F & \text{(finiteness of model)}\\
        &\text{ iff }\exists 1\leq i\leq n:\exists\bar{m}\in M:||\bigland\limits_{j=1}^mR_j(\bar{x_j})||^M_{v[\bar{x}\mapsto\bar{m}]}\in F & \text{(def ~\ref{def:interplat})}\\
        &\text{ iff }\exists 1\leq i\leq n:\exists\bar{m}\in M:\text{ for all atomic subformula }R_j\subseteq\theta_i\; ||R_j(\bar{m_j})||\in F
        &\text{ (def ~\ref{def:interplat})}.
    \end{align*}\qedhere
\end{proof}

With this we can quickly prove the easy direction of the f.p.p.t.
\begin{lemma}\label{protoprevecon}
    Let $\CP$ be a predicate language and $\psi$ an e.$\land$.p-sentence. $Mod_{fin}(\psi)$ is closed under $\rightarrow_p$.
\end{lemma}
\begin{proof}
    Suppose $(A,F,M)\models\psi$ and $(A,F,M)\rightarrow_p(B,G,N)$, i.e.\ $\exists g\colon(A,M)\rightarrow (B,N)$ where $g$ is a protomorphism. By the previous lemma $\exists 1\leq i\leq n:\exists\bar{m}\in M$ such that for all atomic $R_j\subseteq\theta_i$ $R_j^M(\bar{m})\in F$. Therefore for all atomic $R_j\subseteq\theta_i$ $R_j^N(g(\bar{m_j}))\geq 1_B$ and by the previous lemma again $(B,G,N)\models\psi$. 
\end{proof}

This of course immediately implies that $e$.$\land$.p-sentences are preserved under homomorphisms. In the fixed algebra case any protomorphism is also a homomorphism when paired with the identity map and so the converse is also true, that is any formula is preserved under protomorphisms iff it is preserved under homomorphisms. In the mixed algebra case the relationship between protomorphisms and homomorphisms proper is more complicated, we will briefly touch on this in our remarks on future study in section ~\ref{Conclusion}.

The viability of the theorem is the motivation behind insisting that the subset $F$ is a \emph{prime} filter. Requiring $F$ to simply be a filter and the corresponding behaviour of $\land$ and $\lor$ with respect to $F$ made in def ~\ref{def:interplat} can be seen as inherent to these connective being a conjunction/disjunction, the ask for $F$ to be prime with respect to $\lor$ forms a more substantive restriction and one may wonder if it can be dropped. Its necessity can be easily demonstrated; without it we lose viability.

\begin{example}\label{primerequirement}
    Let $A$ be the $4$ element Boolean algebra considered as an interpreting lattice where $F=\{\top\}$ which we label as follows:\begin{figure}[h]
    \centering
    \begin{tikzpicture}[node distance=10mm]
    \node[label={[shift={(0.0,-0.80)}]$\bot$}] (1) at (0.0,0.0) {$\cdot$};
    \node[label={[shift={(-0.25,-0.45)}]$a$}] (2) at (-1.0,1.0) {$\cdot$};
    \node[label={[shift={(0.25,-0.45)}]$b$}] (3) at (1.0,1.0) {$\cdot$};
    \node[label={[shift={(0.0,-0.20)}]$\top$}] (4) at (0.0,2.0) {$\cdot$};
    \draw [](1)--(2); 
    \draw [](1)--(3);
    \draw [](2)--(4);
    \draw [](3)--(4);
    \end{tikzpicture}
\end{figure}

    Taking $\CP=\{R\}$ we define two $\CP$-models over $A$ $\CM_1$ and $\CM_2$ with the same domain $M=\{x,y\}$ and \begin{align*}
        &R^{M_1}:\begin{cases}
            x \mapsto a\\
            y \mapsto b
        \end{cases} & R^{M_2}:\begin{cases}
            x \mapsto a\\
            y \mapsto \bot
        \end{cases}
    \end{align*}
    Then $\CM_1\rightarrow_p\CM_2$ via the identity but if we consider the e.$\land$.p-sentence $\exists zR(z)$ we can easily check that: \begin{align*}
        &\CM_1\models\exists zR(z) & \CM_2\not\models\exists zR(z).
    \end{align*}
\end{example}

When introducing interpreting lattices we allowed for the algebraic signature to be richer and possibly include connectives beyond $\lor$ and $\land$ that act as kinds of disjunction and conjunction. This is the case for models defined over UL-chains or $l$-monoids which come with the additional notion of strong conjunction $\&$ interpreted by the monoid operation. The presence of such a connective lets us define existential-$\&$-positive sentences analogously and also a sentences which combine both forms of conjunction.
\begin{definition}\label{epsentences}
    Let $A$ be an algebra in signature $\CL$ such that $\{\land,\lor,\&\}\in\CL$. Let $\CP$ be a predicate language. We say a $\CP$-sentence
    \begin{enumerate}
        \item $\&$-primitive formula iff $\phi=\exists\bar{x}\bigodot\limits_{i=1}^nR_i(\bar{x_i})$.
        \item $\land$-$\&$-primitive formula (or primitive-positive) iff $\phi=\exists \bar{x}\bigland\limits_{i=1}^n\bigodot\limits_{j=1}^m R_{i,j}(\bar{x}_{i,j})$.
        \item existential-$\&$-positive iff $\phi=\biglor\limits_{i=1}^m\psi_i$ where each $\psi_i$ is a $\&$.p-sentence.
        \item existential-positive iff $\phi=\biglor\limits_{i=1}^m\psi_i$ where each $\psi_i$ is a p.p-sentence. 
    \end{enumerate}
\end{definition}

Similarly, when defining models over either $c$-semirings or semirings the intended interpretation of the conjunction is the monoid operation of the algebra. In this sense, they can be thought of as being defined in the algebraic signature $\{\lor,\&,0,1\}$. In the general these conjunctions may not behave well with respect to the filter $F$ that defines validity. To take the example of UL-chains, we recall $F$ is the upset of $1$ in the lattice order and can observe that homomorphisms (and by extension protomorphisms) fail to preserve the validity of sentences involving $\&$. 

\begin{example}
    Let $A$ be the standard UL-chain on $[0,1]$, that is we take $\land=min$, $\lor=max$ and $\&=\ast\colon[0,1]^2\rightarrow [0,1]$ by \[x\ast y\begin{cases}
        min(x,y) & \text{if } y\leq 1-x\\
        max(x,y) & \text{o.w}.
    \end{cases}\]
    The unit of $\ast$ is $\frac{1}{2}$. 
    
    We consider $\CP=\{R,P\}$ where both $R$ and $P$ are unary relation symbols and define two models $M_1$ and $M_2$ over $A$ both with domain $\{m\}$ and \begin{align*}
        &R^{M_1}(m)=\frac{5}{6} &P^{M_1}(m)=\frac{4}{6} \\
        &R^{M_2}(m)=\frac{1}{3} &P^{M_2}(m)=\frac{1}{6}
    \end{align*}
    Then $M_1\rightarrow M_2$ via the double identity map, but we can make the following calculations to show that the $e.p$-sentence $\exists x\exists yR(x)\amp P(y)$ is not preserved.\begin{align*}
        ||\exists x\exists yR(x)\amp P(y)||^{M_1}=||R(m)||^{M_1}\ast||P(m)||^{M_1}=\frac{5}{6}&\ast\frac{1}{3}=\frac{5}{6}\geq \frac{1}{2}\\
        ||\exists x\exists yR(x)\amp yP(y)||^{M_2}=||R(m)||^{M_2}\ast||P(m)||^{M_2}=\frac{4}{6}&\ast\frac{1}{6}=\frac{1}{6}<\frac{1}{2}.
\end{align*}
\end{example}

Taking the reducts of this example to the appropriate languages yields similar examples for $l$-monoids, c-semirings and semirings. There is more we can say about the status of $e.\&.p$-sentences and $e.p$-sentences - we will return to this topic in Section~\ref{ResiCon}.

%----------------------------------------------------------------------------------------

\section{Finite Homomorphism Preservation}\label{FinHPT}

Let us turn to our core discussion - proving the harder direction of our preservation theorems. The heart of our proof strategy is a recognition that Rossman's classical proof is fundamentally about the behaviour of models with minimal interaction with the underlying logic. Similarly, as we demonstrated in the previous section, protomorphisms and e.$\land$.p-sentences are fundamentally concerned with whether certain atomic formula are true or not; the specifics of the associated algebra do not matter. This allows us to translate between $\CP$-models over interpreting lattices and a classical counterpart in such a way that the behaviour with regards to protomorphisms and e.$\land$.p-sentences is preserved and then apply Rossman's result to the classical translations before pulling back to the our own setting. 

This translation can be introduced in a few different ways. Here it is convenient to present it as a $\CP$-model over the 2 element Boolean algebra $\{\top,\bot\}$ which by inspection is an interpreting lattice when paired with $F=\{\top\}$.
\begin{definition}
    Let $(A,F,M)$ be a $\CP$-model over an interpreting lattice $(A,F)$. We define the $\CP$-model $(\{\top,\bot\},\{\top\},M^{\top})$, also denoted simply as $M^{\top}$ as follows:\[R^{M^{\top}}(\bar{m})=\begin{cases}
        \top & \text{ if }R^M(\bar{m})\in F\\
        \bot & \text { if }R^M(\bar{m})\not\in F.
    \end{cases}\] 
\end{definition}
\begin{remarks}
    In the case where $R$ is a 0-ary relation symbol then it takes no argument $\bar{m}$ from the model and accordingly $R^{M^{\top}}=\top$ iff $R^M\in F$ and $R^{M^{\top}}=\bot$ iff $R^M\not\in F$.
    
    The critical feature of the model just defined is that it is both a many-valued $\CP$-model and also simply a classical $\CP$-model (utilising characteristic functions). This allows us to apply results from classical finite model theory such as Rossman's. Of course, these results have a specific meaning for $M^{\top}$ so we do need to understand how they apply to our context. We make two observations. Firstly, if we consider two $\CP$-models $(A,F,M)$ and $(B,G,N)$ any back-and-forth system between $M^{\top}$ and $N^{\top}$ as defined in definition ~\ref{fixedb&fsys} is also a classical back-and-forth system between them viewed as classical models. We will continue to use the notation $\isom^n_f$ and $\isom_f$ to state the existence of such a back-\&-forth system viewed in either sense. Secondly, a straightforward induction establishes that for any sentence $\phi$ constructed from only $\lor,\land,\exists$ and $\forall$ we have  we have \[M^{\top}\models\phi\text{ iff }M^{\top}\models_{cl} \phi.\]
\end{remarks}

The first basic relationship between a $\CP$-model and its classical counterpart is the precise version of the aforementioned mirroring for both protomorphisms and e.$\land$.p-sentences.
\begin{lemma}\label{clastranshomequiv}
    Let $(A,F,M)$ be a $\CP$-model. Then $(A,F,M)\leftrightarrows_p M^{\top}$ and for any e.$\land$.p-sentence $\psi$ $(A,F,M)\models\psi$ iff $M^{\top}\models\psi$. 
\end{lemma}
\begin{proof}
    We consider the domain map $id_M\colon M\rightarrow M$ and check it is a protomorphism from $(A,F,M)$ into $M^{\top}$ and vice versa. Indeed for any $R\in\CP$ and $\bar{m}\in M$ $R^M(\bar{m})\in F$ iff $R^{M{\top}}(\bar{m})=\top$.

    The second claim follows immediately by lemma ~\ref{protoprevecon}.
\end{proof}

The second relationship revolves around a concept central to Rossman's proof, the idea of $n$-homomorphisms which are themselves defined using the Gaifman graph of a structure~\cite[Section 2\&3]{Rossman08}. 
\begin{definition}
    Let $(A,F,M)$ be a $\CP$-model. We define its Gaifman graph, denoted $G(A,F,M)$ as the graph with vertex set $M$ and $(m,n)\in E$ iff $\exists R\in \CP:\exists\bar{x}\in M:m,n\in\bar{x}$ and $R^M(\bar{x})\in F$.

    We define the tree-depth of $(A,F,M)$, denoted $td(A,F,M)$ to be the tree-depth of $G(A,F,M)$.
\end{definition}
\begin{remarks}
    It is immediately clear that $G(A,F.M)=G(M^{\top})$, with the latter also being the Gaifman graph of $M^{\top}$ when viewed as a classical structure. 
\end{remarks}

We use this to define a form of protomorphism relative to tree-depth.
\begin{definition}
    Let $(A,F,M),(B,G,N)$ be $\CP$-models. We say $(A,F,M)$ $n$-protomorphically maps to $(B,G,N)$, denoted $(A,F,M)\rightarrow^n_p(B,G,N)$ iff for all $\CP$-models $(C,H,S)$ with $td(C,H,S)\leq n$ $(C,H,S)\rightarrow_p(A,F,M)$ implies $(C,H,S)\rightarrow_p(B,G,N)$. 
\end{definition}

We could equivalently define $\rightarrow^n_p$ via the translated structures directly.
\begin{lemma}\label{clastransnhom}
    Let $(A,F,M),(B,G,N)$ be $\CP$-models. $(A,F,M)\rightarrow^n_p(B,G,N)$ iff $M^{\top}\rightarrow^n_{cl}N^{\top}$, that is iff for all \textit{classical} models $C:td(C)\leq n$ $C\rightarrow M^{\top}$ implies $C\rightarrow N^{\top}$.
\end{lemma}
\begin{proof}
    For the only if direction; let $(A,F,M)\rightarrow^n_p(B,G,N)$ and let $S$ be a classical $\CP$-model with tree-depth $\leq n$ such that $S\rightarrow M^{\top}$. We can consider $S$ as a $\CP$-model $(\{\top,\bot\},\{\top\},S)$ over the 2 element Boolean algebra and the map $S\rightarrow M^{\top}$ also witnesses $(\{\top,\bot\},\{\top\},S)\rightarrow_p(A,F,M)$. $td(\{\top,\bot\},\{\top\},S)=td(S)\leq n$ therefore by assumption $(\{\top,\bot\},\{\top\},S)\rightarrow_p(B,G,N)$ and this also witnesses that $S\rightarrow N^{\top}$ as required. 

    For the if direction; letting $M^{\top}\rightarrow^n_{cl}N^{\top}$ and $(C,H,S)\rightarrow_p(A,F,M)$ with $td(C,H,S)\leq n$ then $S^{\top}\rightarrow M^{\top}$ and $td(S^{\top})\leq n$ so $S^{\top}\rightarrow N^{\top}$ and this also witnesses $(C,H,S)\rightarrow_p(B,G,N)$.
\end{proof}

With the two comparisons in place we can immediately prove a many-valued version of Rossman's main combinatorial result - the finite extendability theorem.
\begin{theorem}[Finite Proto Extendability Theorem]\leavevmode\\
    Let $(A,F,M),(B,G,N)$ be $\CP$-models and $n\in\omega$ be such that $(A,F,M)\leftrightarrows^{p(n)}_p(B,G,N)$\footnote{The function p(n) is defined explicitly by Rossman~\cite[Remark 5.3]{Rossman08}, although it can be taken as any 'sufficiently increasing' function}. Then there exists two $\CP$-models $(\{\top,\bot\},\{\top\},\tilde{M})$ and $(\{\top,\bot\},\{\top\},\tilde{N})$ such that:
    \[(A,F,M)\leftrightarrows_p(\{\top,\bot\},\{\top\},\tilde{M})\equiv_n(\{\top,\bot\},\{\top\},\tilde{N})\leftrightarrows_p(B,G,N).\]
\end{theorem}
\begin{proof}
    By lemmas ~\ref{clastranshomequiv} \& ~\ref{clastransnhom} we have $(A,F,M)\leftrightarrows_p M^{\top}$, $(B,G,N)\leftrightarrows_p N^{\top}$ and $M^{\top}\leftrightarrows^{p(n)} N^{\top}$. Applying Rossman's finite extendability theorem~\cite[Corollary 5.14]{Rossman08} for classical structures $\exists\tilde{M},\tilde{N}$ such that $\tilde{M}\isom^n_f\tilde{N}$, $M^{\top}\leftrightarrows\tilde{M}$ and $N^{\top}\leftrightarrows\tilde{N}$ (as classical structures). Taking composition of the following gives us the first and last equivalence we need: \[(A,F,M)\leftrightarrows_pM^{\top}\leftrightarrows_p(\{\top,\bot\},\{\top\},\tilde{M})\;\;(\{\top,\bot\},\{\top\},\tilde{N})\leftrightarrows_pN^{\top}\leftrightarrows_p(B,G,N).\]
    
    For the last equivalence, viewing $\tilde{M}$ and $\tilde{N}$ as $\CP$-models $(\{\top,\bot\},\{\top\},\tilde{M})$ and $(\{\top,\bot),\{\top\},\tilde{N})$ we have $(\{\top,\bot\},\{\top\},\tilde{M})\isom^n_f(\{\top,\bot\}),\{\top\},\tilde{N})$. Then by theorem ~\ref{fixedfinteb&f} $(\{\top,\bot\},\{\top\},\tilde{M})\equiv_n (\{\top,\bot\},\{\top\},\tilde{N})$. 
\end{proof}

To prove our initial versions of the finite h.p.t.\ we need to get a little deeper into the proof of Rossman's result itself. First we quote the well-known association with finite structures and $\land$-primitive formulas, namely that for any classical finite structure $M$ there is a $\land$-primitive sentence $\theta_M$ such that $qd(\theta_M)=td(M)$ and for any classical structure $N$ $M\rightarrow N$ iff $N\models\theta_M$~\cite[Lemma 2.14]{Rossman08}. Secondly, we utilise the concept of an $n$-core of a classical structure~\cite[Definition 3.10, Lemma 3.11]{Rossman08}. Rossman develops a fair amount of theory around $n$-cores, but for our purposes we just need that for any classical structure $M$ and its $n$-core $Core^n(M)$ the following properties hold:\begin{align*}
    M\leftrightarrows^n Core^n(M)\:\:\: & td(Core^n(M))\leq n,\\
    M\rightarrow^n N \text{ iff } Core^n(M)\rightarrow N &\text{ for any classical structure }N.
    \end{align*}

We use our translated models to prove new versions of the two main components to Rossman's result. The first is a variant of~\cite[Lemma 3.16]{Rossman08}. 
\begin{lemma}\label{hardlemma1}
    Let $P$ and $Q$ be classes of $\CP$-models over interpreting lattices (in the same signature $\CL)$ such that for any pair of $\CP$-models over interpreting lattices $(A,F,M),(B,G,N)$, $(A,F,M)\in P$ and $(A,F,M)\rightarrow^n_p(B,G,N)$ implies $(B,G,N)\in Q$.
    Then there is an e.$\land$.p-sentence of quantifier depth $n$ $\psi$ such that $P\subseteq Mod_{fin}(\psi)\subseteq Q$.
\end{lemma}
\begin{proof}
    We consider the set $P^{\top}=\{M^{\top}:M\in P\}$ and let $\mathbb{P}$ be the set of $n$-cores of structures in $P^{\top}$. This is finite so let $\{C_n\}_{i=1}^{m}$ be its members and consider the $\land$-primitive formula sentences $\theta_i$ corresponding to each $C_n$. We note that each $\theta_i$ has quantifier depth $n$ and for any classical structure $D$ $C_i\rightarrow D$ iff $D\models_{cl}\theta_i$. We let $\psi:=\biglor_{1\leq i\leq m}\theta_i$. 

    Now consider $(A,F,M)\in P$. $M^{\top}\in P^{\top}$ and $Core^n(M^{\top})\in\mathbb{P}$, i.e.\ $Core^n(M^{\top})=C_i$ for some $i\leq m$. A basic property of $n$-cores is that $Core^n(M^{\top})\rightarrow M^{\top}$~\cite[Definition 3.10]{Rossman08} and so $M^{\top}\models_{cl}\theta_i$ and consequently $M^{\top}\models\theta_i$ and moreover $M^{\top}\models\psi$. Then by lemma ~\ref{clastranshomequiv} $(A,F,M)\models\psi$ and we have $P\subseteq Mod_{fin}(\psi)$

    Let $(B,G,N)$ be a $\CP$-model such that $(B,G,N)\models\psi$. By lemma ~\ref{clastranshomequiv} again $N^{\top}\models\psi$ so  $N^{\top}\models\theta_i$ for some $i\leq m$. Therefore, $N^{\top}\models_{cl}\theta_i$ and $C_i\rightarrow N^{\top}$. Then $C_i\in\mathbb{P}$ so $C_i=Core^n(M^{\top})$ for some $M^{\top}\in P^{\top}$ with $(A,F,M)\in P$. From~\cite[Definition 3.10]{Rossman08} we have $Core^n(M^{\top})\rightarrow N^{\top}$ implies $M^{\top}\rightarrow^nN^{\top}$ and so by lemma ~\ref{clastransnhom} $(A,F,M)\rightarrow^n_p(B,G,N)$. Therefore, by our initial assumption $(B,N)\in Q$ and $Mod_{fin}(\psi)\subseteq Q$ as required.  
\end{proof}

The second main component to Rossman's proof is using finite extendability to check the sufficiency condition just established~\cite[Theorem 5.15]{Rossman08}
\begin{lemma}\label{hardlemma2}
    Let $P$ and $Q$ be classes of $\CP$-models over interpreting lattices (in the same signature $\CL$) and $\phi$ a $\CL,\CP$-sentence with quantifier depth at most $n$ such that for any $\CP$-models $(A,F,M),(B,G,N)$ we have:\begin{itemize}
        \item $(A,F,M)\in P$, $(A,F,M)\rightarrow_p(B,G,N)$ implies $(B,G,N)\models\phi$.
        \item $(A,F,M)\models\phi$, $(A,F,M)\rightarrow_p(B,N)$ implies $(B,G,N)\in Q$. 
    \end{itemize}
    Then $(A,F,M)\in P$ and $(A,F,M)\rightarrow^{p(n)}(B,G,N)$ implies $(B,G,N)\in Q$.
\end{lemma}
\begin{proof}
    Let $(A,F,M)\in P$ and $(A,FM)\rightarrow^{p(n)}_p (B,G,N)$. By lemmas ~\ref{clastranshomequiv} and ~\ref{clastransnhom}, \[(A,F,M)\leftrightarrows_p M^{\top}\rightarrow^{p(n)}_{cl} N^{\top}\leftrightarrows_p(B,G,N).\] We consider $Core^{p(n)}(M^{\top})$, for which $M^{\top}\leftrightarrows^{p(n)}Core^{p(n)}(M^{\top})$ and $Core^{p(n)}(M^{\top})\rightarrow N^{\top}$. Applying finite extendability to $M^{\top}$ and $Core^{p(n)}(M^{\top})$ we obtain $\tilde{M^\top}$ and $\tilde{Core^{p(n)}(M^{\top})}$ such that $\tilde{M^{\top}}\equiv_n\tilde{Core^{p(n)}(M^{\top})}$. Putting this together gives the following diagram:
\[\begin{tikzcd}
                            &[-25pt]                  &[-25pt] \tilde{M^{\top}} &[-25pt] \equiv_n                     &[-25pt] \tilde{Core^{p(n)}(M^{\top})} &[-25pt]               &[-25pt]                    &[-25pt]                    &[-25pt]                             \\[-25pt]
                            &                  & \uparrow\downarrow         &                              & \uparrow\downarrow                      &               &                    &                    &                             \\[-25pt]
(A,F,M) & \leftrightarrows_p & M^{\top}         & \leftrightarrows^{p(n)}_{cl} & Core^{p(n)}(M^{\top})         & \rightarrow_p & N^{\top} & \leftrightarrows_p & (B,G,N)
\end{tikzcd}\]

    Tracing through our diagram we obtain that $(B,G,N)\in Q$:\begin{gather*} (A,F,M)\in P \text{ and }(A,F,M)\rightarrow_p\tilde{M^{\top}} \text{ implies } \tilde{M^{\top}}\models\phi.\\ \tilde{M^{\top}}\models\phi, qd(\phi)\leq n \text{ and } \tilde{M^{\top}}\equiv_n\tilde{Core^{p(n)}(M^{\top})} \text{ implies } \tilde{Core^{p(n)}(M^{\top})}\models \phi.\\
    \tilde{Core^{p(n)}(M^{\top})}\models\phi\text{ and }\tilde{Core^{p(n)}(M^{\top})}\rightarrow_p(B,G,N)\text{ implies }(B,G,N)\in Q.
    \end{gather*}
\end{proof}

We can now give the proofs of our first two preservation theorem; the fixed case as a corollary.
\begin{proof}[Proof of Finite Protomorphism Preservation Theorem]\leavevmode\\
    Our if direction is lemma ~\ref{protoprevecon}, so let $\phi$ be a $\CP$-sentence preserved under protomorphisms. Taking $P=Mod_{fin}(\phi)=Q$ we immediately have for any $\CP$-models $(A,F,M)$ and $(B,G,N)$ $(A,F,M)\in P$ and $(A,F,M)\rightarrow_p(B,G,N)$ implies $(B,G,N)\models \phi$ and $(A,F,M)\models\phi$ and $(A,F,M)\rightarrow_p(B,G,N)$ implies $(B,G,N)\in Q$ so by lemma ~\ref{hardlemma2} for any two $\CP$-models $(A,F,M)\in P$ and $(A,F,M)\rightarrow^{p(n)}(B,G,N)$ implies $(B,G,N)\in Q$. Therefore, by lemma ~\ref{hardlemma1} there is an e.$\land$.p-sentence $\psi$ with quantifier depth $p(n)$ such that $Mod_{fin}(\phi)=P\subseteq Mod_{fin}(\psi)\subseteq Q=Mod_{fin}(\phi)$, i.e.\ $Mod_{fin}(\phi)=Mod_{fin}(\psi)$ as required. 
\end{proof}

\begin{proof}[Proof of Fixed Finite Homomorphism Preservation Theorem]\leavevmode\\
    By fixing an interpreting lattice $(A,F)$ and only considering models over $(A,F)$ all protomorphism are homomorphism when paired with the identity map $id_A$. Then lemma ~\ref{protoprevecon} gives $i.\Rightarrow iii.$, $iii.\Rightarrow ii.$ is immediate from the above comment and $ii.\Rightarrow i.$ is the previous theorem restricted to models over $(A,F)$.  
\end{proof}

%--------------------------------------------------------------------------------------

\section{Strong Conjunction and Integrality}\label{ResiCon}
At this point we want to make a direct comparison with the preservation theorems for finite models we've just established and the preceding work on such preservation theorems. Recall that in~\cite{DellundeVidal19} the authors give a compactness based proof of a homomorphism preservation theorem for all (infinite and finite) structures interpreted over a fixed finite MTL-chain. Following our proof strategy for the finite case one can give a homomorphism preservation theorem for infinite structures that is compactness free. This is possible because alongside proving classical finite homomorphism preservation Rossman also proved an 'equi-rank' version of the standard classical homomorphism preservation theorem~\cite[Theorem 4.12]{Rossman08} and accordingly we can utilise this in exactly the same way we did the finite version. An important difference is that if we take our models over arbitrary interpreting lattices as we did in the finite case we lose that $\exists$ is witnessed, i.e. for a $\CP$-formula $\exists x\phi(x)$ it is possible that $||\exists x\phi(x)||^M\not=||\phi(m)||^M$ for all $m\in M$. To recover this we insist that we only consider models over finite interpreting lattices; our ability to work with possibly infinite structures comes at the cost of reducing the class of algebras we work over\footnote{This restriction and the reasoning for it is also recognised by Dellunde and Vidal in their work ~\cite[pg 5]{DellundeVidal19}}. The avoidance of compactness is not just an interesting methodological feature, it means that we obtain a theorem that still holds for our quite general interpreting lattices even in the absence of a compactness theorem for these structures. Accordingly, we get the following two theorems.
\begin{theorem}[Protomorphism Preservation Theorem]\leavevmode\\
Let $K$ be the class of finite interpreting lattices. Let $\CP$ be a predicate language and $\phi$ a consistent $\CP$-sentence over $K$, i.e $Mod^K(\phi)\not=\varnothing$.

Then $\phi$ is equivalent in the finite to an e.$\land$.p-sentence $\psi$ iff $\phi$ is preserved under protomorphisms. That is there exists an e.$\land$.p-sentence $\psi: nr(\phi)=nr(\psi)$ and $Mod^K(\phi)=Mod^K(\psi)$ iff $Mod^K(\phi)$ is closed under $\rightarrow_p$.
\end{theorem}

\begin{theorem}[Fixed Homomorphism Preservation Theorem]\leavevmode\\
Let $\CP$ be a predicate language, $A$ a finite interpreting lattice and $\phi$ a consistent $\CP$-sentence over $A$ (i.e.\ $Mod^A(\phi)\not=\varnothing)$. The following are equivalent:\begin{enumerate}[label=\roman*.]
    \item $\phi$ is equivalent over $A$ to an e.$\land$.p-sentence $\psi$ with the same quantifier depth. i.e.\ there is an e.$\land$.p-sentence $\psi:qd(\phi)=qd(\psi)$ and  $Mod^A(\phi)=Mod^A(\psi)$.
    \item $\phi$ is preserved under protomorphisms on $A$, i.e.\ $Mod^A(\phi)$ is closed under $\rightarrow_p$.
    \item $\phi$ is preserved under homomorphisms on $A$, i.e.\ $Mod^K(\phi)$ is closed under $\rightarrow$.
\end{enumerate}
\end{theorem}

These theorems are not strictly a generalisation of the one presented by Dellunde and Vidal~\cite[Theorem 11.3]{DellundeVidal19} because theirs incorporates the residuated conjunction, establishing a homomorphism preservation theorem relative to existential-positive sentences as defined in definition ~\ref{epsentences}. As it stands, our result seems to mostly say that the 'classical part' of residuated lattices behave as we would expect, which whilst interesting is perhaps less compelling than an understanding of the residuated connectives. In the general setting we've worked in so far we have already seen why we must exclude residuated conjunction. The critical difference relates to the property of \emph{integrality} for residuated lattice. In this context we can recover good behaviour for $\&$ with respect to morphisms (either proto- or homo-). We say that a residuated interpreting lattice is \textit{integral} iff the unit of the residuated conjunction $1$ is the top element in the lattice, i.e.\ $\forall a\in A$ $a\leq 1$. Consequently, the filter $F$ defining truth becomes just $\{1\}$. The significant consequence of integrality is that it implies that the residuated conjunction behaves just as lattice conjunction with respect to truth, this lets us establish that e.$\&$.p-sentences and consequently e.p-sentences are preserved under $\rightarrow_p$.
\begin{lemma}\label{resiconjmetalogic}
    Let $A$ be an integral residuated interpreting lattice. $\forall a,b\in A$ $a\amp b=1$ iff $a=1$ and $b=1$.
\end{lemma}
\begin{proof}
    Follows easily by monotonicty of $\&$. $a\leq 1$ implies $a\amp b\leq 1\amp b=b$ and symmetrically $a\amp b\leq a$. Then, $1=a\amp b\leq a$ implies $a=1$ and similarly $b=1$. The reverse is trivial.  
\end{proof}

\begin{lemma}
    Let $\CP$-predicate language and $(A,M)$ be a $\CP$-model over an integral residuated interpreting lattice $A$ and $\psi=\biglor\limits_{i=1}^n\theta_i$ an e.p-sentence. $(A,M)\models\psi$ iff $\exists 1\leq i\leq n:\exists\bar{m}\in M:$ for all atomic subformula $R_{j,k}\subseteq \theta_i$ $R^M_{j,k}(\bar{m_{j,k}})=1$ where $\bar{m_{j,k}}$ is the subsequence of $\bar{m}$ corresponding to the subsequence $\bar{x_{j,k}}$ of $\bar{x}$. 

    Moreover, $Mod_{fin}(\psi)$ is closed under $\rightarrow_p$. 
\end{lemma}
\begin{proof}
    Following the proof of lemma ~\ref{easylemma} we obtain that $(A,M)\models\psi$ iff $\exists 1\leq i\leq n:\exists\bar{m}\in M:\forall 1\leq j\leq n_1$ $||\bigodot\limits_{k=1}^{n_2} R_{j,k}(\bar{x_{j,k}})||_{v[\bar{x}\mapsto\bar{m}]}=1$. Then by the previous lemma, the latter holds iff $\exists 1\leq i\leq n:\exists\bar{m}\in M:$ for all atomic subformula $R_{j,k}\subseteq\theta_i$ $||R_{j,k}(\bar{x_{j,k}})||_{v[\bar{x}\mapsto\bar{x}]}=1$ iff $\exists 1\leq i\leq n:\exists\bar{m}\in M:$ for all atomic subformula $R_{j,k}\subseteq\theta_i$ $R^M_{j,k}(\bar{m_{j,k}})=1$.

    For the moreover, letting $(A,M)\models\psi$ and $(A,M)\rightarrow_p(B,N)$ where $A,B$ are integral residuated interpreting lattice. We have that $\exists 1\leq i\leq n:\exists\bar{m}\in M:$ for all atomic subformula $R_{j,k}\subseteq\theta_i$ $R^M_{j,k}(\bar{m_{j,k}})=1$. Letting $g\colon (A,M)\rightarrow (B,N)$ be a protomorphism from $(A,M)$ to $(B,N)$, by the definition of protomorphisms $R^N_{j,k}(g(\bar{m_{j,k}}))$, so taking $g(\bar{m})$ and $i$ as witness we have by our previous part that $(B,N)\models\psi$. 
\end{proof}

With this recovered we can easily prove preservation theorem extended to cover e.p-sentences. 
\begin{theorem}[Finite Integral Protomorphism Preservation Theorem]\leavevmode
    Let $K$ be the class of integral residuated interpreting lattices. Let $\CP$ be a predicate language and $\phi$ a consistent $\CP$-sentence in the finite (i.e.\ $Mod^K_{fin}(\phi)\not=\varnothing)$. Then $\phi$ is equivalent over $K$ in the finite to an e.p-sentence $\psi$ iff $\phi$ is preserved under protomorphisms over $K$. That is there exists an e.p-sentence $\psi:Mod^K_{fin}(\phi)=Mod_{fin}(\psi)$ iff $Mod^K_{fin}(\phi)$ is closed under $\rightarrow_p$.

    Moreover, if $A$ is an integral residuated interpreting lattice and $\phi$ a consistent $\CP$-sentence over $A$ (i.e.\ $Mod^{A}_{fin}(\phi)\not=\varnothing$) then the following are equivalent:\begin{enumerate}[label=\roman*.]
    \item $\phi$ is equivalent over $A$ in the finite to an e.p-sentence $\psi$.
    \item $\phi$ is preserved under protomorphisms on $A$, i.e.\ $Mod^A_{fin}(\phi)$ is closed under $\rightarrow_p$.
    \item $\phi$ is preserved under homomorphisms on $A$, i.e.\ $Mod^A_{fin}(\phi)$ is closed under $\rightarrow$.
    \end{enumerate}
\end{theorem}
\begin{proof}
    The only if direction is just the last claim of our previous lemma. For the if direction by the finite p.p.t.\ theorem $\phi$ is equivalent to an e.$\land$.p-sentence $\psi$ which is of course also an e.p-sentence. The moreover follows similarly from the fixed finite h.p.t.\ theorem and the previous lemma upon noting as usual that any protomorphism between models defined over $A$ is automatically a homomorphism when paired with the identity map. 
\end{proof}

This additional restriction also gives the direct comparison to the work done by Dellunde and Vidal. MTL-chains are integral residuated interpreting lattices so the preceding theorem specialises to give a direct answer to the status of finite preservation they asked. Additionally, we can combine the integral restriction with our prior 'equi-rank' protomorphism preservation theorem to obtain an equivalent theorem for (possibly infinite) models defined over a finite integral residuated interpreting lattice.

The restriction to integral algebras can be replicated for an alternative treatment of our problem inspired by the framework for VCSP by Bistarelli, Montanari and Rossi in~\cite{BistarelliMontanariRossi97}. Their work provides the most general treatment of VCSPs and takes what are known as c-semirings as their valuation structure. 
\begin{definition}
    A \textit{c-semiring} is tuple algebra $A=\langle A, +,\times,0,1\rangle$ such that\begin{itemize}
        \item $+\colon\mathcal{P}(A)\rightarrow A$ is defined over (possibly infinite) sets of elements of $A$ where \begin{itemize}
            \item $+(\{a\})=a$,
            \item $+(\varnothing)=0$ $+(A)=1$,
            \item $+(\bigcup_{i\in S}A_i)=+(\bigcup\limits_{i\in S}+(A_i))$
        \end{itemize}
        \item $\times$ is a binary operation that is associative and commutative with unit element $1$ and absorption element $0$.
        \item $\times$ distributives over $+$, that is for any $a\in A$ and $B\subseteq A$ $a\times +(B)=+(\{a\times b\in A:b\in B\})$.
    \end{itemize}
\end{definition}
Importantly, c-semirings are bounded complete lattices where we take $+$ as the generalised \textit{infinite} disjunction with $1$ and $0$ acting as the top and bottom element of the lattice. Strictly speaking they are not interpreting lattices for any choice of $F$ as the $\times$ operation does not capture lattice conjunction. However, $\times$ is monotone with respect to the resulting lattice order and its unit element $1$ is the top element with respect to the lattice order. Together this implies $\times$ is \emph{intensive} in c-semirings i.e. $a\times b\leq a$ and $a\times b\leq b$ for all $a,b\in A$) that together with integrality means we have the nice behaviour of $\times$ when taking $F=\{1\}$. In general, the unit $1$ is not prime for the additive operator $+$ in c-semirings, by in practice all the major VCSP formalisms make use of specific c-semirings for which this the case\footnote{For example, all three of the more specific contexts consider in~\cite[Section 6]{BistarelliMontanariRossi97}.} Taking all this together, we have an equivalent to lemma~\ref{resiconjmetalogic} for the c-semirings of interest. Thus, one can rerun our entire proof but working with existential-$\&$-sentences as defined in definition ~\ref{epsentences} and replacing $\&$ by $\land$ to define the classical counterparts yielding:
\begin{theorem}[Finite c-semiring Protomorphism Theorem]\leavevmode\\
    Let $K$ the class of c-semirings such that $A\in K$ $\forall a,b\in A$ $a+b=1$ iff $a=1$ or $b=1$. Let $\CP$ be a predicate language and $\phi$ a consistent $\CP$-sentence in the finite over $K$ (i.e.\ $Mod^K_{fin}(\phi)\not=\varnothing)$. Then $\phi$ is equivalent over $K$ in the finite to an e.$\&$.p-sentence $\psi$ iff $\phi$ is preserved under protomorphisms over $K$. That is there exists an e.$\&$.p-sentence $\psi:Mod^K_{fin}(\phi)=Mod^K_{fin}(\psi)$ iff $Mod^K_{fin}(\phi)$ is closed under $\rightarrow_p$.

    Moreover, if $A\in K$ and $\phi$ a consistent $\CP$-sentence over $A$ (i.e.\ $Mod^{A}_{fin}(\phi)\not=\varnothing$) then the following are equivalent:\begin{enumerate}[label=\roman*.]
    \item $\phi$ is equivalent over $A$ in the finite to an e.$\&$.p-sentence $\psi$.
    \item $\phi$ is preserved under protomorphisms on $A$, i.e.\ $Mod^A_{fin}(\phi)$ is closed under $\rightarrow_p$.
    \item $\phi$ is preserved under homomorphisms on $A$, i.e.\ $Mod^A_{fin}(\phi)$ is closed under $\rightarrow$.
    \end{enumerate}
\end{theorem}

Given this result for c-semirings (albeit a restricted class of them) it is natural to consider the case for semiring semantics. In example ~\ref{3maincases} we observed that the min-max semiring is an interpreting lattice when paired with any upset and so our main theorem ~\ref{FinHPT} applies directly. Slightly more generally, for any lattice semiring provided an equivalent result to lemma ~\ref{resiconjmetalogic} holds relative to an appropriate choice of $F$, we could in a similar manner to the c-semiring case run the proof strategy to recover a particular preservation theorem. Beyond this however our proof strategy breaks down due to the differing interpretation of the quantifiers $\exists$ and $\forall$ in semiring semantics. As noted in remarks ~\ref{quantifierremark}, in semiring semantics the quantifiers are interpreted as generalisations of the disjunction and conjunction connectives $+$ and $\cdot$ respectively. When outside of lattice semirings this divergence has significant implications for basic model theoretic results. In particular, as discussed in ~\cite{BrinkeGradelMrkonjic23} the status of back-and-forth systems for semirings is quite different with the equivalent result to theorem ~\ref{fixedfinteb&f} failing in most cases.
%-------------------------------------------------------------

\section{Conclusions and Further Study}\label{Conclusion}
In this paper we have begun the investigation on preservation theorems for finite many-valued models. Inspired by the successful generalisation of the homomorphism preservation theorem into first-order MTL by Dellunde and Vidal and Rossman's recovery of the same theorem in classical finite models, we utilised the latter's result to identify and prove the appropriate version of the theorem in the many-valued setting - our finite protomorphism preservation theorem. There remains a lot of work to be done on the topic of many-valued preservation theorems. In the infinite setting, the \L os--Tarski theorem (linking universal sentences with strong embeddings) and Chang--\L o\'{s}--Suszko theorem (linking $\forall_2$-formulas with unions of chains) have been recovered for models defined over a fixed finite MTL-algebra~\cite{BadiaCostaDellundeNoguera19}. This was the same setting as the (infinite) homomorphism preservation theorem by Dellunde and Vidal~\cite{DellundeVidal19} and given our expansion of that theorem to models defined over interpreting lattices it prompts whether something similar is possible for the other preservation theorems. In the finite setting less is known, the \L os--Tarski theorem and Lyndon positivity~\cite[Theorem 1.4,1.5]{Rossman08} fail for classical finite models which would suggest they also fail in the many-valued setting but this requires conformation. 

As referenced throughout, there is a general interest in the application of many-valued model theory to questions of computational complexity, a first approach can be found in~\cite{HorcikMoraschiniVidal17}, and classically the finite homomorphism preservation theorem is linked to the constraint satisfaction problem (CSP)~\cite{Rossman08}. Our main theorem and the ancillary discussion in section~\ref{ResiCon} cover many of the algebraic structures that VCSP problems consider, and we can naturally ask what the impact the finite protomorphism preservation theorem may have. The failure of residuated conjunction preservation in non-integral settings could be of particular interest for the feasibility of VCSP investigations in those contexts. 

To close we comment on two more areas of future study in a little more detail. First is the status of the residuated  conjunction for non-integral residuated. We've already seen the that residuated conjunction is not preserved by homomorphisms without integrality. However, there is a natural class of maps between structures that do preserve residuated conjunction even in non-integral settings - monomorphisms. Recall from definition ~\ref{othermorph} that a monomorphism is a homomorphism $(f,g)$ that preserves the order of interpreted relation symbols, i.e. $f(R^M(\bar{m}))\leq R^N(g(\bar{m}))$.

At present, we are only able to show the viability of preservation of $e.p$-sentences by monomorphisms for $\exists$-\textit{witnessed} models, that is models where for any $\CP$-formula $\phi(\bar{x}):=\exists x\psi(x,\bar{x})$ and any $\bar{m}\in M$ there is some $m\in M$ such that $||\exists x\psi(x,\bar{m})||^M=||\psi(m,\bar{m})||$.
\begin{lemma}
   Let $K$ be the class of integral residuated interpreting lattices. Let $\CP$ be a predicate language, $\psi$ and $e.p$-sentence. The models of $\psi$ are closed under monomorphisms.
\end{lemma}
\begin{proof}
    Suppose $(A,M)\models\psi$ and $(f,g)\colon (A,M)\rightarrow(B,N)$ a monomorphism. We claim that $f(||\psi||^M)\leq ||\psi||^N$, in that case $(A,M)\models\psi$ implies $1_A\leq ||\psi||^M$ and so $1_B\leq f(||\psi||^M)\leq ||\psi||^N$ and $(B,N)\models\psi$ as required. For the claim we prove by induction that for any $e.p$-formula $\phi(\bar{x})$ and $\bar{m}\in M$ $f(||\phi(\bar{m})||^M)\leq ||\phi(g(\bar{m}))||^N$.

    The base case where $\phi$ is atomic is simply the definition of monomorphism. The inductive step for $\lor$,$\land$ and $\&$ all follow by their monotonicity in the algebra, e.g. with $\&$ we have by monotonicity that $\forall a,b,c,d\in A$ that $a\leq b$ and $c\leq d$ imply $a\& c\leq b\& c\leq b\& d$ and so assuming that $f(||\psi(\bar{m})||^M\leq ||\psi(g(\bar{m}))||^N$ $f(||\lambda(\bar{m})||^M\leq ||\psi(g(\bar{m})||^N$ we obtain $f(||\psi(\bar{m}) \& ||\lambda(\bar{m})||^M)$ = $f(||\psi(\bar{m})||^M)\& f(||\lambda(\bar{m})||^M)\leq ||\psi(g(\bar{m}))||^N\&||\psi(g(\bar{m})||^N=||\psi\&\lambda(g(\bar{m}))||^N$. This just leaves the case for $\exists$. Considering $\exists x\psi(x,\bar{x})$ we letting $\bar{m}$ be arbitrary and apply $\exists$-witnessed to find $m\in M$ such that $||\exists x\psi(x,\bar{m})||^M=||\psi(m,\bar{m})||^M$. Then applying $f$ and the inductive hypothesis we obtain:
    \[ f(||\exists x\psi(x,\bar{m})||^M)=f(||\psi(m,\bar{m})||^M)\leq ||\psi(g(m),g(\bar{m})||^N\leq ||\exists x\psi(x,g(\bar{m})||^N.\]
\end{proof}

Unfortunately, when it comes to attempting a proof of the full monomorphism preservation theorem (infinite or finite) our easy strategy of applying Rossman's result no longer works, there is no way to define a classical counterpart to a structure on the two element Boolean algebra that has mutual monomorphism available with the original structure. We believe that the most viable strategy would require a full recovery of Rossman's original classical proof with appropriate generalisations to the many-valued context of the concepts contained within. Alternatively, there has recently been work investigating the classical homomorphism preservation theorem alongside other kinds of preservation theorems via categorical methods~\cite{AbramskyReggio23} which could potentially be adapted and applied to many-valued models. 

Another relationship we want to understand better is the one between the existence of a protomorphism and homomorphism between two given structures and whether there are significant consequences for when one is only able to define a protomorphism between structures. Recall that any protomorphism $g\colon (A,F,M)\rightarrow(B,G,N)$ is a homomorphism when paired with \textit{any} $\CL$-algebra homomorphism $f\colon A\rightarrow B$ and accordingly the only difference between the $\rightarrow_p$ and $\rightarrow$ relations is the existence of \textit{some} homomorphism $f\colon A\rightarrow B$. Any information of the map $f$ seems to have no direct consequence on the structures themselves, its action is irrelevant to the comparison in atomic formulas the map $g$ enables making its presence seem unnecessary. At the same time, our proof for preservation between structures defined over a class of algebras really does only go through for protomorphisms. This prompts the question of whether there really is some genuine difference (at the logical level) between the two relations or whether at the weakest level of homomorphism we are essentially free to ignore our algebra component and solely work with protomorphisms.

We can say a small amount about the most straightforward case. We could bluntly require that for our class of algebras $K$ that for any  (non-trivial) $A,B\in K$ there exists a homomorphishm $f\colon A\rightarrow B$. In this case the two notions collapse. This situation has had some prior attention, namely for quasivarieties of finite type with at least one finite non-trivial member the condition is equivalent to a property called \textit{passive structural completeness} (PSC)~\cite[Theorem 7.6]{MoraschiniRafteryWannenburg20}. The restriction to quasivarieties of finite type is not major, all the classes of algebras we are interested in are subclasses of a given variety or quasivariety. In such a situation we can obtain the following corollary to our main theorem.
\begin{corollary}
    Let $K$ be a class of interpreting lattices contained within a PSC quasivariety containing a finite non-trivial member. Then the relations of $\rightarrow_p$ and $\rightarrow$ for models over $K$ coincide and in particular any $\CP$-sentence $\phi$ is is preserved under $\rightarrow_p$ over $K$ iff $\phi$ is preserved under $\rightarrow$ over $K$.
\end{corollary}

The class of FSI Heyting algebras provide an example of this type. The variety of Heyting algebras is PSC~\cite[Example 7.9]{MoraschiniRafteryWannenburg20} and FSI Heyting algebras are interpreting lattices when taking $F=\{1\}$.

\section*{Acknowledgments}

We were partially supported by the Australian Research Council grant DE220100544.

%-------------------------------------------------------------------
%   Bibliography
%-------------------------------------------------------------------

\printbibliography[heading=bibintoc]

@article{AbramskyReggio23,
    title = {Arboreal categories: An axiomatic theory of resources},
    author = {S. Abramsky and L. Reggio},
    journal = {Logical Methods in Computer Science},
    volume = {19},
    number = {3},
    pages = {14:1-14:36},
    year = {2023},
}

@article{AjtaiGurevich87,
    author = {M. Ajtai and Y. Gurevich},
    title = {Monotone versus positive},
    journal = {J. ACM},
    year = {1987},
    volume = {34},
    pages = {1004-1015},
}

@article{BadiaCostaDellundeNoguera19,
    author = {G. Badia and V. Costa and P. Dellunde and C. Noguera},
    title = {Syntactic characterizations of classes of first-order structures in mathematical fuzzy logic},
    journal = {Soft Computing},
    year = {2019},
    volume = {23},
    pages = {2177-2186},
}

@article{BadiaNoguera22,
  author={G. Badia and C. Noguera},
  journal={IEEE Transactions on Fuzzy Systems}, 
  title={A 0-1 Law in Mathematical Fuzzy Logic}, 
  year={2022},
  volume={30},
  number={9},
  pages={3833-3840},
}

@article{BistarelliMontanariRossi97,
    author = {S. Bistareli and U. Montanari and F. Rossi},
    title = {Semiring-based constraint satisfaction and optimization},
    journal = {Journal of the ACM},
    volume = {44},
    number = {2},
    pages = {201--236},
    year = {1997},
}

@InProceedings{BiziereGradelNaaf23,
  author =	{C. Bizi\`{e}re and E. Gr\"{a}del and M. Naaf},
  title =	{Locality Theorems in Semiring Semantics},
  booktitle =	{48th International Symposium on Mathematical Foundations of Computer Science (MFCS 2023)},
  pages =	{20:1--20:15},
  series =	{Leibniz International Proceedings in Informatics (LIPIcs)},
  year =	{2023},
  volume =	{272},
  editor =	{Leroux, J\'{e}r\^{o}me and Lombardy, Sylvain and Peleg, David},
  publisher =	{Schloss Dagstuhl -- Leibniz-Zentrum f{\"u}r Informatik},
  address =	{Dagstuhl, Germany},
}

@misc{BrinkeGradelMrkonjic23,
      title={Ehrenfeucht-Fra\"iss\'e Games in Semiring Semantics}, 
      author={S. Brinke and E. Grädel and L. Mrkonjić},
      year={2023},
      eprint={2308.04910},
      archivePrefix={arXiv},
      primaryClass={cs.LO},
      url={https://arxiv.org/abs/2308.04910}, 
}

@article{ChandraMerlin77,
    author = {A. Chandra and P. Merlin},
    title = {Optimal implementation of conjunctive queries in relational databases},
    journal = {Proceedings of the ninth annual ACM sympsium on Theory of computing},
    year = {1977},
    pages = {77-90},
}

@book{CintulaNoguera21,
    author = {P. Cintula and C. Noguera},
    title = {Logic and Implication},
    publisher = {Springer},
    year = {2021},
}

@article{Dellunde12,
    title = {Preserving mappings in fuzzy predicate logics},
    journal = {Journal of Logic and Computation},
    volume = {22},
    number = {6},
    pages = {1367-1389},
    author = {P. Dellunde},
    year = {2012},
}

@article{DellundeGarciaNoguera16,
    author = {P. Dellunde A. Garc\'{i}a-Cerdaña and C. Noguera},
    title = {L\"{o}wenheim-Skolem theorems for non-classical first-order algebraizable logics},
    journal = {Logic Journal of the IGPL},
    year = {2016},
    volume = {24},
    number = {3},
    pages = {321-345},  
}

@article{DellundeGarciaNoguera18,
    title = {Back-and-forth systems for fuzzy first order models},
    journal = {Fuzzy Sets and Systems},
    volume = {345},
    number = {},
    pages = {},
    author = {P. Dellunde and A. Garc\'{i}a-Cerdaña and C. Noguera},
    year = {2018},
}

@article{DellundeVidal19,
    title = {Truth-preservation under fuzzy pp-formulas},
    journal = {International Journal of Uncertaintiy, Fuzzyness and Knowledge-Based Systems},
    volume = {27},
    number = {},
    pages = {89--105},
    author = {P. Dellunde and A. Vidal},
    year = {2019},
}

@book{EbbinghausFlum95,
    author = {H.D. Ebbinghaus and J. Flum},
    title = {Finite Model Theory},
    publisher = {Springer},
    year = {1995},
}

@inproceedings{GradelHelalNaafWilke22,
    author = {Gr\"{a}del, Erich and Helal, Hayyan and Naaf, Matthias and Wilke, Richard},
    title = {Zero-One Laws and Almost Sure Valuations of First-Order Logic in Semiring Semantics},
    year = {2022},
    isbn = {9781450393515},
    publisher = {Association for Computing Machinery},
    address = {New York, NY, USA},
    url = {https://doi.org/10.1145/3531130.3533358},
    doi = {10.1145/3531130.3533358},
    booktitle = {Proceedings of the 37th Annual ACM/IEEE Symposium on Logic in Computer Science},
    articleno = {41},
    numpages = {12},
    location = {Haifa, Israel},
    series = {LICS '22}
}

@article{Green11,
    author = {T.J. Green},
    title = {Containment of conjunctive queries on annotated relations},
    journal = {Theory of Computing Systems},
    year = {2011},
    volume = {49},
    pages = {429-459},
}

@inproceedings{GreenTannen17,
    author = {Green, Todd J. and Tannen, Val},
    title = {The Semiring Framework for Database Provenance},
    year = {2017},
    publisher = {Association for Computing Machinery},
    address = {New York, NY, USA},
    url = {https://doi.org/10.1145/3034786.3056125},
    doi = {10.1145/3034786.3056125},
    booktitle = {Proceedings of the 36th ACM SIGMOD-SIGACT-SIGAI Symposium on Principles of Database Systems},
    pages = {93–99},
    series = {PODS '17}
}

@InProceedings{Gurevich84,
    author= {Y. Gurevich},
    editor= {B{\"o}rger, Egon and Oberschelp, Walter and Richter, Michael M. and Schinzel, Brigitta and Thomas, Wolfgang}, 
    title= {Toward logic tailored for computational complexity},
    booktitle= {Computation and Proof Theory},
    year= {1984},
    publisher= {Springer Berlin Heidelberg},
    address= {Berlin, Heidelberg},
    pages= {175--216},
}

@book{Handbook,
    editor = {P. Cintula and P. H\'{a}jek and C. Noguera},
    title = {Handbook of Mathematical Fuzzy Logic Vol 1},
    publisher = {College Publications},
    year = {2011},
}

@book{Hodges97,
    author = {W. Hodges},
    title = {A shorter model theory},
    publisher = {Cambridge University Press},
    year = {1997},
}

@article{HorcikMoraschiniVidal17,
    title = {An Algebraic Apporach to Valued Constraint Satisifaction},
    journal = {26th EACSL Annual Conference on Computer Science Logic},
    volume = {82},
    number = {},
    pages = {},
    author = {R. Hor\u{c}\'ik and T. Moraschini and A. Vidal},
    year = {2017},
}

@article{Lyndon59,
    author = {R. Lyndon},
    title = {Properties preserved under homomorphism},
    journal = {Pacific J. Math},
    year = {1959},
    volume = {9},
    pages = {129-142},
}

@article{MoraschiniRafteryWannenburg20,
    title = {Single generated quasivarities and residuated structures},
    journal = {Mathematical Logic Quarterly},
    volume = {66},
    number = {},
    pages = {150-172},
    author = {T. Moraschini and J. Raftery and J.J. Wannenburg},
    year = {2020},
}

@article{Noguera:Thesis,
    author = {C. Noguera},
    title = {Algebraic study of axiomatic extensions of triangular norm based fuzzy logics},
    journal = {Monographs of the Artificial Intelligence Research Institute},
    volume = {27},
    year = {2007},
}

@article{NolaGerla86,
    author = {A. Di Nola and G. Gerla},
    title = {Fuzzy models of first-order languages},
    journal = {Zeitshrift f\"{u}r Mathematische Logik und Grundlagen der Mathematik},
    volume = {32},
    pages = {19--24},
    year = {1986},
}

@article{Tait59,
    author = {W. Tait} ,
    title = {A counterexample to a conjecture of Scott and Suppes},
    journal = {J. Symbolic Logic},
    year = {1959},
    volume = {24},
    pages = {15-16},
}

@article{Rossman08,
    title = {Homomorphism preservation theorems},
    journal = {Journal of the ACM},
    volume = {55},
    number = {3},
    pages = {},
    author = {B. Rossman},
    year = {2008},
}

\end{document}